\documentclass[a4paper,11pt,reqno]{amsart} 
\usepackage{amsmath,amssymb,amsfonts,amsthm}
\usepackage[tmargin=30mm,bmargin=30mm,hmargin=25mm]{geometry}
\usepackage{enumerate}
\usepackage{mathrsfs}
\usepackage[T1]{fontenc}
\usepackage{stmaryrd}
\usepackage[dvipsnames]{xcolor}
\usepackage{amscd}
\usepackage[all]{xy}
\usepackage{tikz}
\usetikzlibrary{arrows,decorations.pathmorphing,decorations.pathreplacing}
\tikzset{black/.style={circle,fill=black,inner sep=3pt,outer sep=3pt},white/.style={circle,fill=white,draw=black,inner sep=3pt,outer sep=3pt},b/.style={draw,fill=black,inner sep=3pt,outer sep=3pt},w/.style={draw,fill=white,inner sep=3pt,outer sep=3pt}}
\usepackage[colorlinks, linkcolor=blue!72,anchorcolor=orange,
citecolor=red,urlcolor=Emerald, bookmarksopen,bookmarksdepth=2]{hyperref}
\usepackage{cleveref}

\newcommand{\A}{\mathcal{A}}
\newcommand{\B}{\mathcal{B}}
\newcommand{\C}{\mathcal{C}}
\newcommand{\D}{\mathcal{D}}

\newcommand{\E}{\mathcal{E}}
\newcommand{\F}{\mathcal{F}}
\newcommand{\clH}{\mathcal{H}}
\newcommand{\clS}{\mathcal{S}}

\newcommand{\T}{\mathcal{T}}
\newcommand{\clU}{\mathcal{U}}
\newcommand{\V}{\mathcal{V}}
\newcommand{\X}{\mathcal{X}}
\newcommand{\Y}{\mathcal{Y}}
\newcommand{\add}[1]{\operatorname{add}({#1})}
\newcommand{\Db}[1]{D^b({#1})}

\newcommand{\Dl}[1]{\mathcal{D}^{\leqslant{#1}}}
\newcommand{\Dg}[1]{\mathcal{D}^{\geqslant{#1}}}
\newcommand{\Ul}[1]{\mathcal{U}^{\leqslant{#1}}}
\newcommand{\Ug}[1]{\mathcal{U}^{\geqslant{#1}}}
\newcommand{\Vl}[1]{\mathcal{V}^{\leqslant{#1}}}
\newcommand{\Vg}[1]{\mathcal{V}^{\geqslant{#1}}}
\newcommand{\Xl}[1]{\mathcal{X}^{\leqslant{#1}}}
\newcommand{\Xg}[1]{\mathcal{X}^{\geqslant{#1}}}
\newcommand{\Yl}[1]{\mathcal{Y}^{\leqslant{#1}}}
\newcommand{\Yg}[1]{\mathcal{Y}^{\geqslant{#1}}}

\newcommand{\Hm}{m\text{-}\mathcal{H}}

\newcommand{\Em}{m\text{-}\mathcal{E}}
\newcommand{\id}{\mathrm{id}}
\DeclareMathOperator{\tors}{\mathsf{tors}}

\DeclareMathOperator{\stors}{\mathsf{s-tors}}

\DeclareMathOperator{\tstr}{\mathsf{t-str}}
\newcommand{\Cone}{\operatorname{Cone}}
\newcommand{\Cocone}{\operatorname{Cocone}}
\renewcommand{\mod}{\operatorname{mod}}
\theoremstyle{definition}
\newtheorem{Lemma}{Lemma}[section]
\newtheorem{Thm}[Lemma]{Theorem}

\newtheorem{Coro}[Lemma]{Corollary}
\newtheorem{Def}[Lemma]{Definition}
\newtheorem{Prop}[Lemma]{Proposition}
\newtheorem{Remark}[Lemma]{Remark}
\newtheorem{Not}[Lemma]{Notation}
\newtheorem{Ex}[Lemma]{Example}
\numberwithin{equation}{section}
\title[Intervals of torsion pair and generalized HRS tilting]{Intervals of torsion pairs and generalized Happel-Reiten-Smal{\o} tilting}
\author{Jieyu Chen}
\author{Zengqiang Lin$^*$}

\thanks{The authors were supported by the National Natural Science Foundation of China (Grant No.
	12471035) and by the Natural Science Foundation of Fujian Province (Grant No. 2024J01088)}
\thanks{$^*$Corresponding author}
\thanks{E-mail address: zqlin@hqu.edu.cn}

\begin{document}
	
	\begin{abstract}
		Let $\A$ be an abelian category with a torsion pair $(\T,\F)$. Happel-Reiten-Smal{\o} tilting provides a method to construct a new abelian category $\B$ with a torsion pair associated to $(\T,\F)$, which is exactly the heart of a certain $t$-structure on the bounded derived category $\Db\A$. In this paper, we mainly study generalized HRS tilting. We first show that an interval of torsion pairs in extriangulated categories with negative first extensions is bijectively associated with torsion pairs in the corresponding heart, which yields several new observations in triangulated categories. Then we obtain a generalization of HRS tilting by replacing hearts of $t$-structures with extended hearts. As an application, we show that certain $t$-structures on triangulated subcategories can be extended to $t$-structures on the whole triangulated categories.
	\end{abstract}
	
	\maketitle
	
	\section{Introduction}
	The notion of torsion pairs in abelian categories was first introduced by Dickson \cite{Dickson66}, and the triangulated version was studied by Iyama and Yoshino \cite{IY08}.  Be\u{\i}linson, Bernstein, and Deligne \cite{BBD82} introduced the notion of t-structures on triangulated categories to study perverse sheaves. The study of torsion pairs and $t$-structures has been of fundamental importance to homological algebra, representation theory, algebraic geometry, and derived categories. A landmark result in this direction is the correspondence established by Happel, Reiten, and Smal{\o} \cite{HRS96}, now known as HRS tilting. Let $\A$ be an abelian category with a torsion pair $(\T,\F)$. Let $\Db\A$ be the derived category of $\A$. They constructed an abelian subcategory $\B \subseteq \Db\A$ such that $(\F[1],\T)$ is a torsion pair in $\B$. Their main idea is to construct a $t$-structure on $\Db\A$ with heart $\B$ by the torsion pair $(\T,\F)$. The process from $(\A, (\T,\F))$ to $(\B, (\F[1],\T))$ is called classical HRS tilting. 
	
	The construction of $t$-structures on triangulated categories through torsion pairs in abelian subcategories has since been generalized and refined in various directions, see \cite{BR07,Polishchuk07,Tattar21}. Recently, Adachi, Enomoto, and Tsukamoto \cite{AET23} extended the HRS tilting framework to the setting of extriangulated categories equipped with a negative first extension. They introduced the notion of $s$-torsion pairs, which is a common generalization of t-structures on triangulated categories and torsion pairs in abelian categories, and showed that  an interval in the poset of $s$-torsion pairs is bijectively associated with $s$-torsion pairs in the corresponding heart; see \cite[Theorem 3.9]{AET23}.
	
	Since $s$-torsion pairs are torsion pairs satisfying one additional condition, our first goal is to investigate the intervals of torsion pairs and show that the bijections in \cite[Theorem 3.9]{AET23} are consequences of more general bijections in extriangulated categories. Let $\C$ be an extriangulated category with a negative first extension $\mathbb{E}^{-1}$. For two torsion pairs $t_1 = (\clU_1, \V_1)$ and $t_2 = (\clU_2, \V_2)$ in $\C$, we define a relation $t_1 \preccurlyeq t_2$ by requiring $\C(\clU_1,\V_2)=0$ and $\mathbb{E}^{-1}(\clU_1,\V_2)=0$. The interval $\tors[t_1, t_2]$ consists of all torsion pairs $t=(\clU,\V)$ such that $t_1 \preccurlyeq t \preccurlyeq t_2$. Our first main result establishes a bijection between the torsion pairs in such an interval and certain torsion pairs in the heart $\clH_{[t_1,t_2]}=\clU_2\cap \V_1$, which is also an extriangulated category with a negative first extension. 
	
	\begin{Thm}[\Cref{theorem3.8}]\label{theorem1.1}
		Let $\C$ be an extriangulated category with a negative first extension. Let $t_1=(\clU_1,\V_1)$ and $t_2=(\clU_2,\V_2)$ be two torsion pairs in $\C$ with $t_1 \preccurlyeq t_2$. There exist order preserving, mutually inverse bijections between
		\begin{enumerate}
			\item the set of torsion pairs (resp. $s$-torsion pairs) $t=(\clU,\V)$ in $\C$ with $t_1 \preccurlyeq t \preccurlyeq t_2$, 
			\item the set of torsion pairs (resp. $s$-torsion pairs)  $(\T,\F)$ in $\clH_{[t_1,t_2]}$ with $\mathbb{E}^{-1}(\T,\V_2)=0$ and $\mathbb{E}^{-1}(\clU_1,\F)=0$.
		\end{enumerate}
	\end{Thm}
	
	In particular, triangulated categories can be viewed as extriangulated categories with negative first extensions. In this case, $(\clU_1,\V_1)\preccurlyeq (\clU_2,\V_2)$ if $\clU_1 \subseteq \clU_2$ and $\clU_1[1]\subseteq\clU_2$. We obtain the following as direct consequences.
	
	\begin{Coro}[\Cref{corollary3.10}, \Cref{corollary3.11}]\label{corollary1.2}
		Let $\D$ be a triangulated category with shift functor $[1]$. Let $(\clU_1,\V_1)$ and $(\clU_2,\V_2)$ be two torsion pairs such that $(\clU_1,\V_1)\preccurlyeq (\clU_2,\V_2)$. Let $(\C_1^{\leqslant 0},\C_1^{\geqslant 0})$ and $(\C_2^{\leqslant 0},\C_2^{\geqslant 0})$ be two $t$-structures such that $\C_1^{\leqslant 0}\subseteq\C_2^{\leqslant 0}$. Then there are order preserving, mutually inverse bijections
		\begin{enumerate}
			\item between the set of torsion pairs $(\clU,\V)$ in $\D$ with $(\clU_1,\V_1)\preccurlyeq(\clU,\V)\preccurlyeq (\clU_2,\V_2)$ and the set of torsion pairs $(\T,\F)$ in $\clU_2\cap\V_1$ with $\T[1]\subseteq\clU_2$ and $\F[-1]\subseteq \V_1$;
			\item between the poset of $t$-structures $(\D^{\leqslant 0},\D^{\geqslant 0})$ on $\D$ with $\clU_1\subseteq\D^{\leqslant 0}\subseteq\clU_2$ and the poset of $s$-torsion pairs $(\T,\F)$ in $\clU_2\cap\V_1$ with $\T[1]\subseteq\clU_2$ and $\F[-1]\subseteq \V_1$;
			\item between the poset of torsion pairs $(\clU,\V)$ in $\D$ with $\C_1^{\leqslant 0}\subseteq\clU\subseteq \C_2^{\leqslant 0}$ and the poset of torsion pairs in $\C_2^{\leqslant 0}\cap\C_1^{\geqslant 1}$;
			\item between the poset of $t$-structures $(\D^{\leqslant 0},\D^{\geqslant 0})$ on $\D$ with $\C_1^{\leqslant 0}\subseteq\D^{\leqslant 0}\subseteq \C_2^{\leqslant 0}$ and the poset of $s$-torsion pairs in $\C_2^{\leqslant 0}\cap\C_1^{\geqslant 1}$.
		\end{enumerate}
	\end{Coro}
	
	Recently, J{\o}rgensen \cite{Jorgensen21} replaced hearts of $t$-structures by proper abelian subcategories of triangulated categories and provided a generalization of  HRS tilting; see \cite[Theorem B]{Jorgensen21}. Let $\A$ be a proper abelian subcategory with a torsion pair $(\T,\F)$. He showed that, under certain conditions, $\F[1]*\T$ is a proper abelian subcategory with a torsion pair $(\F[1],\T)$. He also proved that there exist bijections between the torsion pairs $(\T,\F)$ in $\A$ and certain proper abelian subcategories restricted by $\A$.  Zhou \cite{Zhou24} introduced the notion of extended hearts in triangulated categories and gave another generalization of HRS tilting; see \cite[Theorem 0.1]{Zhou24}. Let $(\Ul0,\Ug0)$ be a $t$-structure with heart $\clH$. Let $m$ be a positive integer, then the $m$-extended heart is defined as 
	$$\Hm = \Ug {-(m-1)} \cap \Ul0=\clH[m-1]*\clH[m-2]*\cdots*\clH,$$ 
	which is an extriangulated category with a negative first extension. Let $(\T,\F)$ be an $s$-torsion pair in $\Hm$, Zhou proved that $\F[m]*\T$ is an $m$-extended heart with an $s$-torsion pair $(\F[m],\T)$. Our second goal is to prove that there exist bijections between the $s$-torsion pairs $(\T,\F)$ in $\Hm$ and certain $m$-extended hearts restricted by $\Hm$, thereby giving a generalization of HRS tilting for extended hearts.
	
	\begin{Thm}[\Cref{theorem4.8}]\label{theorem1.3}
		Let $\D$ be a triangulated category with shift functor $[1]$ and $\Hm$ be an $m$-extended heart. Then there are order preserving, mutually inverse bijections between
		\begin{enumerate}
			\item the poset of $s$-torsion pairs in $\Hm$,
			\item the poset of $m$-extended hearts $\Em$ with $\Em \subseteq \Hm[m]*\Hm$ and $\Hm \subseteq \Em*\Em[-m]$.
		\end{enumerate}
	\end{Thm}
	
	Our third goal is to investigate the conditions under which a $t$-structure on a triangulated subcategory can be extended to a $t$-structure on the whole triangulated category. By \Cref{corollary1.2} (4), we have the following third theorem, which refines and reformulates earlier work on $t$-structures and generalizes \cite[Theorem 5.1]{CLZ23}.
	
	\begin{Thm}[\Cref{theorem5.3}]\label{theorem1.4}
		Let $\D$ be a triangulated category with shift functor $[1]$ and $(\Ul0,\Ug0)$ be a $t$-structure with heart $\clH$. Let $\clS$ be a triangulated subcategory of $\D$ such that $(\clS\cap\Ul0,\clS\cap\Ug0)$ is a $t$-structure on $\clS$ and $\clH\subseteq\clS$. Then there are order preserving, mutually inverse bijections between
		\begin{enumerate}
			\item the poset of t-structures $(\X^{\leqslant 0},\X^{\geqslant 0})$ on $\D$ with $\Ul{-m}\subseteq\X^{\leqslant 0}\subseteq \Ul0$,
			\item the poset of t-structures $(\Y^{\leqslant 0},\Y^{\geqslant 0})$ on $\clS$ with $\clS\cap\Ul{-m}\subseteq\Y^{\leqslant 0}\subseteq \clS\cap\Ul0$.
		\end{enumerate}
	\end{Thm}
	
	The article is organized as follows. In Section 2, we recall the necessary background material on extriangulated categories with negative first extensions, torsion pairs, $s$-torsion pairs, and $t$-structures. In section 3, we study the intervals of torsion pairs and prove \Cref{theorem1.1}. In Section 4, we generalize HRS tilting on extended hearts and prove \Cref{theorem1.3}. In the final section, we discuss the extensions of $t$-structures and prove \Cref{theorem1.4}.
	
	\section{Preliminaries}
	Throughout this paper, we assume that each category is skeletally small, that is, the isomorphism classes of objects form a set. All subcategories are assumed to be full and closed under isomorphisms. In this section, we introduce some basic definitions and facts that will be needed later.
	
	\subsection{Extriangulated categories with negative first extensions}
	
	We will use the notion of an extriangulated category from \cite{NP19} without recalling the complete definition. An \emph{extriangulated category} consists of a triple $(\C,\mathbb{E},\mathfrak{s})$ satisfying certain axioms, where $\C$ is an additive category, $\mathbb{E}:\C^{\mathrm{op}}\times \C \rightarrow Ab$ is an additive bifunctor and $\mathfrak{s}$ is an additive realization of $\mathbb{E}$, which maps an \emph{$\mathbb{E}$-extension} $\delta \in \mathbb{E}(C,A)$ to an equivalence class of pairs of morphisms $[A\stackrel{x}{\longrightarrow}B\stackrel{y}{\longrightarrow}C]$. In this case, we call $A\stackrel{x}{\longrightarrow}B\stackrel{y}{\longrightarrow}C\stackrel{\delta}{\dashrightarrow}$ an \emph{$\mathbb{E}$-triangle}.
	
	\begin{Ex} \label{example2.1}
		\begin{enumerate}
			\item[] 
			\item A triangulated category $\D$ with shift functor $[1]$ becomes an extriangulated category $(\D,\mathbb{E},\mathfrak{s})$ by the following data.
			\begin{enumerate}[(i)]
				\item $\mathbb{E}(-,-):=\D(-,-[1])$.
				\item  For any $A,C \in \D$, take an arbitrary $\mathbb{E}$-extension $\delta \in \mathbb{E}(C,A)$ and a distinguished triangle $A\stackrel{x}{\longrightarrow} B\stackrel{y}{\longrightarrow}C\stackrel{\delta}{\longrightarrow} A[1]$, and let $\mathfrak{s}(\delta)=[A\stackrel{x}{\longrightarrow}B \stackrel{y}{\longrightarrow}C]$. 
			\end{enumerate}
			\item An exact category $\E$ becomes an extriangulated category $(\E,\mathbb{E},\mathfrak{s})$ by the following data.
			\begin{enumerate}[(i)]
				\item $\mathbb{E}(-,-):=\operatorname{Ext}^1(-,-)$.
				\item For any $A,C \in \E$ and $\delta=[A\stackrel{x}{\longrightarrow}B \stackrel{y}{\longrightarrow}C]\in \operatorname{Ext}^1(C,A)$, let $\mathfrak{s}(\delta)=\delta$. 
			\end{enumerate}    
		\end{enumerate}
	\end{Ex}
	
	Let $\X$, $\Y$ and $\mathcal{Z}$ be subcategories of an extriangulated category $\C$. Let $\X * \Y$ denote the subcategory of $\C$ consisting of $A$ which admits an $\mathbb{E}$-triangle $X \rightarrow A \rightarrow Y\dashrightarrow$ with $X \in \X$ and $Y \in \Y$. By axioms (ET4) and (ET4)$^\textrm{op}$, it follows that $\X*(\Y*\mathcal{Z})=(\X*\Y)*\mathcal{Z}$. We say $\X$ is closed under extensions if $\X * \X =\X$. Denote by $\Cone(\X,\Y)$ the subcategory of $\C$ consisting of $A$ which admits an $\mathbb{E}$-triangle $X \rightarrow Y \rightarrow A \dashrightarrow$ with $X \in \X$ and $Y \in Y$. We say $\X$ is closed under cones if $\Cone(\X, \X) =\X$. Denote by $\Cocone(\X,\Y)$ the subcategory of $\C$ consisting of $A$ which admits an  $\mathbb{E}$-triangle $A \rightarrow X \rightarrow Y \dashrightarrow$ with $X \in \X$ and $Y \in Y$. We say $\X$ is closed under cocones if $\Cocone(\X, \X) =\X$.
	
	\begin{Def}[{\cite[Definition 2.3]{AET23}}]\label{definition2.2}
		Let $\C=(\C,\mathbb{E},\mathfrak{s})$ be an extriangulated category. A \emph{negative first extension structure} on $\C$ consists of the following data.
		\begin{enumerate}
			\item  $\mathbb{E}^{-1}: \C^{\mathrm{op}}\times \C \rightarrow Ab$ is an additive bifunctor.
			\item For an arbitrary $\mathbb{E}$-triangle $A{\rightarrow}B{\rightarrow}C{\dashrightarrow}$ and an object $W \in \C$, the sequences
			\begin{align*}
				\mathbb{E}^{-1}(W,A) \rightarrow \mathbb{E}^{-1}(W,B) \rightarrow &\mathbb{E}^{-1}(W,C) \rightarrow \C(W,A)\rightarrow \C(W,B),\\
				\mathbb{E}^{-1}(C,W) \rightarrow \mathbb{E}^{-1}(B,W) \rightarrow &\mathbb{E}^{-1}(A,W) \rightarrow \C(C,W)\rightarrow \C(B,W)
			\end{align*}
			are exact.
		\end{enumerate}
	\end{Def}
	
	\begin{Ex}\label{example2.3}
		\begin{enumerate}
			\item[]
			\item Let $\D$ be a triangulated category. Then $\D$ is an extriangulated category with a negative first extension, where $\mathbb{E}^{-1}(-,-)=\D(-,-[-1])$; see \cite[Example 2.4]{AET23}.
			\item Let $\E$ be an exact category. Then $\E$ is an extriangulated category with a negative first extension, where $\mathbb{E}^{-1}(-,-)=0$.
			\item  Let $k$ be a field and $\Lambda$ be a finite-dimensional $k$-algebra with $\operatorname{gl.dim}\Lambda \leqslant n$. Then $\mod \Lambda$  is an extriangulated category with a negative first extension, where $\mathbb{E}^{-1}(-,-)=\operatorname{Ext}^n_\Lambda(-,-)$; see  \cite[Example 3.19]{AET23}.
			\item  Let $\C=(\C,\mathbb{E},\mathfrak{s},\mathbb{E}^{-1})$ be an extriangulated category with a negative first extension and $\C'$ be an extension-closed subcategory of $\C$. By restriction, $\C'$ inherits an extriangulated structure and a negative first extension structure.
		\end{enumerate}
	\end{Ex}
	
	\subsection{Torsion pairs and \texorpdfstring{$t$}{t}-structures}
	
	\begin{Def}[{\cite[Definition 4.1]{Tattar24}}]
		Let $\C=(\C,\mathbb{E},\mathfrak{s})$ be an extriangulated category, and let $\clU$ and $\V$ be additive subcategories of $\C$, which are closed under both direct sums and direct summands. We say that $(\clU,\V)$ is a \emph{torsion pair} in $\C$, where $\clU$ is a \emph{torsion class} and $\V$ a \emph{torsionfree class}, if the following two conditions are satisfied.
		\begin{enumerate}
			\item[(TP1)] $\C(\clU,\V)=0$;
			\item[(TP2)] $\C=\clU * \V$. 
		\end{enumerate}
	\end{Def}  
	
	\begin{Lemma}[{\cite[Remark 4.2]{Tattar24}}]
		Let $\C=(\C,\mathbb{E},\mathfrak{s})$ be an extriangulated category, and $(\clU,\V)$ be a torsion pair in $\C$. Then the following holds.
		\begin{enumerate}
			\item $\clU$ and $\V$ are closed under extensions;
			\item $\clU=\{A \in \C \mid \C(A,\V)=0\}$, $\V=\{A \in \C \mid \C(\clU,A)=0\}$. Therefore, for any subcategory $\mathcal{W}$ of $\C$, we have $\mathcal{W}\subseteq\clU$ if and only if  $\C(\mathcal{W},\V)=0$, and $\mathcal{W}\subseteq\V$ if and only if  $\C(\clU,\mathcal{W})=0$.
		\end{enumerate} 
	\end{Lemma}
	
	\begin{Def}[{\cite[Definition 3.1]{AET23}}]
		Let $\C=(\C,\mathbb{E},\mathfrak{s},\mathbb{E}^{-1})$ be an extriangulated category with a negative first extension. A torsion pair $(\clU,\V)$ is called an \emph{$s$-torsion pair} if it satisfies the additional condition.
		\begin{enumerate}
			\item[(STP)]  $\mathbb{E}^{-1}(\clU,\V)=0$.
		\end{enumerate}
	\end{Def}
	
	Before giving some specific examples of  $s$-torsion pairs, let us recall $t$-structures on triangulated categories.
	
	\begin{Def}[{\cite[Definition 1.3.1]{BBD82}}]
		Let $\D$ be a triangulated category and $(\Ul0,\Ug0)$ a pair of subcategories of $\D$. For any integer $n$, we denote $\Ul{n}=\Ul0[-n]$ and $\Ug{n}=\Ug0[-n]$. We call $(\Ul0,\Ug0)$ a \emph{$t$-structure} on $\D$ if it satisfies the following three conditions.
		\begin{enumerate}
			\item[(t1)] $\D(\Ul0,\Ug1)=0$;
			\item[(t2)] $\D=\Ul0 * \Ug1$;
			\item[(t3)] $\Ul0 \subseteq \Ul1$, $\Ug1 \subseteq \Ug0$.
		\end{enumerate}
		In this case, we call $\Ul0$ an \emph{aisle} and $\Ug0$ a \emph{co-aisle}. The \emph{heart} of the $t$-structure $(\Ul0,\Ug0)$ is $\clH:=\Ul0 \cap \Ug0$. It is well-known that the heart of a $t$-structure is an abelian category; see \cite[Theorem 1.3.6]{BBD82}.
	\end{Def}
	
	Let $\A$ be an abelian category and $\Db\A$ its bounded derived category. Then $(\Ul0,\Ug0)$ is a $t$-structure on $\Db\A$ with heart $\A$, called the \emph{standard $t$-structure}, where
	\begin{align*}
		\Ul0&=\{ X \in \Db\A \mid \mathrm{H}^n(X)=0 \textrm{ for any } n > 0 \},\\
		\Ug0&=\{  X \in \Db\A \mid \mathrm{H}^n( X)=0\textrm{ for any } n < 0 \}.
	\end{align*}
	There are plenty of properties of $t$-structures, see \cite{BBD82,GM,KV88}. For any $t$-structure $\clU=(\Ul0,\Ug0)$ on $\D$, we have $\clU[n]=(\Ul{-n},\Ug{-n})$ is a $t$-structure on $\D$ for any integer $n$. We observe that both aisles and co-aisles of $t$-structures are closed under direct summands and extensions. Aisles are closed under cones and co-aisles are closed under cocones.  The inclusion $\Ul n \hookrightarrow \D$ admits a right adjoint $\tau_{\leqslant n}:\D \rightarrow \Ul n$, and the inclusion $\Ug n \hookrightarrow \D$ admits a left adjoint $\tau_{\geqslant n}:\D \rightarrow \Ug n$. They are called \emph{truncation functors} associated to $\clU$. Both aisles and co-aisles are closed under truncations.
	
	The following examples show that $s$-torsion pairs are a common generalization of $t$-structures on triangulated categories and torsion pairs in exact categories.
	
	\begin{Ex}
		\begin{enumerate}
			\item[]
			\item By regarding a triangulated category $\D$ as an extriangulated category with a negative first extension as before, a pair $(\X,\Y)$ of subcategories is an $s$-torsion pair in $\D$ if and only if $(\X,\Y[1])$ is a t-structure on $\D$.
			\item Let $\E$ be an exact category. By regarding $\E$ as an extriangulated category with a trivial negative first extension as \Cref{example2.3} (2), then the $s$-torsion pairs in $\E$ are exactly the torsion pairs in $\E$.
			\item Let $\Lambda$ be a hereditary algebra. Then $\mod \Lambda$ is an extriangulated category with a negative first extension by \Cref{example2.3} (3). A torsion pair $(\T,\F)$ is an $s$-torsion pair in $\mod \Lambda$ if and only if  $\T$ and $\F$ are Serre subcategories of $\mod \Lambda$; see \cite[Corollary 3.22]{AET23}.
		\end{enumerate}
	\end{Ex}
	
	\section{Intervals of torsion pairs in extriangulated categories}
	Throughout this section, let $\C=(\C,\mathbb{E},\mathfrak{s},\mathbb{E}^{-1})$ be an extriangulated category with a negative first extension. Denote by $\tors\C$ (resp. $\stors\C$) the set of all the torsion pairs (resp. $s$-torsion pairs) in $\C$. Let  $t_1=(\clU_1,\V_1)$ and $t_2=(\clU_2,\V_2)$ be two torsion pairs in $\C$. We define $t_1 \preccurlyeq t_2$ if $\C(\clU_1,\V_2)=0$ and $\mathbb{E}^{-1}(\clU_1,\V_2)=0$. In this case, we denote by $\tors[t_1,t_2]$  (resp. $\stors[t_1,t_2]$) the \emph{interval} of torsion pairs (resp. $s$-torsion pairs) in $\C$ consisting of $t=(\clU,\V)$ with $t_1 \preccurlyeq t \preccurlyeq t_2$.   We call the subcategory $\clH_{[t_1,t_2]}=\V_1\cap\clU_2$ the \emph{heart} of this interval. Since $\V_1$ and $\clU_2$ are extension-closed, $\clH_{[t_1,t_2]}$ is an extriangulated category with a negative first extension. 
	
	Note that the relation $\preccurlyeq$ defined above is transitive but not necessarily reflexive, hence it may not be a partial order in $\tors[t_1,t_2]$. In particular, if $t_1$ or $t_2$ is an $s$-torsion pair, then $t_1 \preccurlyeq t_2$ if and only if $\C(\clU_1,\V_2)=0$, that is, $\clU_1 \subseteq \clU_2$. Assume that both $t_1$ and $t_2$ are $s$-torsion pairs, then $\tors[t_1,t_2]$ consists of all the torsion pairs $t=(\clU,\V)$ in $\C$ such that $\clU_1 \subseteq \clU \subseteq \clU_2$, and $\stors\clH_{[t_1,t_2]}$ is a poset. 
	
	Before stating the main result of this section, we give the following lemmas to show the relations between the torsion pairs in $\C$ and those in $\clH_{[t_1,t_2]}$.
	
	\begin{Lemma}\label{lemma3.1}
		Let  $t_1=(\clU_1,\V_1)$ and $t_2=(\clU_2,\V_2)$ be two torsion pairs in $\C$ with $t_1 \preccurlyeq t_2$. Then $\clU_2= \clU_1 * (\clU_2\cap \V_1)$ and $\V_1 =  (\clU_2\cap \V_1) * \V_2$.
	\end{Lemma}
	\begin{proof}
		We will show the first part of the assertion as the argument for the second part is analogous. For any $X \in \clU_2$, there is an $\mathbb{E}$-triangle
		$$ U_1 \rightarrow X \rightarrow V_1 \dashrightarrow$$
		with $U_1 \in \clU_1$ and $V_1 \in \V_1$. Applying $\C(-,\V_2)$, we obtain an exact sequence 
		$$ \mathbb{E}^{-1}(U_1,\V_2) \rightarrow \C(V_1,\V_2) \rightarrow \C(X,\V_2).$$ 
		Since $\mathbb{E}^{-1}(U_1,\V_2)=0$ and $\C(X,\V_2)=0$, we have $\C(V_1,\V_2)=0$ which implies that $V_1 \in \clU_2$. So $\clU_2 \subseteq \clU_1 * (\clU_2\cap \V_1)$. The converse inclusion follows from the facts that $\clU_1 \subseteq \clU_2$  and $\clU_2$ is closed under extensions. Thus $\clU_2 = \clU_1 * (\clU_2\cap \V_1)$.
	\end{proof}
	
	\begin{Lemma}\label{lemma3.2}
		Let $t=(\clU,\V)$ be a torsion pair in $\C$ with $t_1 \preccurlyeq t \preccurlyeq t_2$. Then the following holds.
		\begin{enumerate} 
			\item $(\clU \cap \V_1 , \V \cap \clU_2)$ is a torsion pair in $\clH_{[t_1,t_2]}$.
			\item $\clU_1*(\clU \cap \V_1)=\clU$ and $(\V \cap \clU_2)*\V_2=\V$.
		\end{enumerate}
	\end{Lemma}
	\begin{proof}
		\par(1)  Since $\clH_{[t_1,t_2]}(\clU \cap \V_1 , \V \cap \clU_2) \subseteq \C(\clU,\V) =0$, (TP1) holds. To verify (TP2), take any $X \in \clH_{[t_1,t_2]}$. There exists an $\mathbb{E}$-triangle
		$$U{\rightarrow}X{\rightarrow}V\dashrightarrow$$
		with $U \in \clU$ and $V \in \V$. Applying $\C(-,\V_2)$ yields an exact sequence:
		$$\mathbb{E}^{-1}(U,\V_2)\rightarrow \C(V,\V_2)\rightarrow \C(X,\V_2).$$
		Since $\mathbb{E}^{-1}(\clU,\V_2)=0$ and $\C(\clU_2,\V_2)=0$, we get $\C(V,\V_2)=0$, that is, $V \in \clU_2$. Hence $V \in \V\cap \clU_2$. Similarly, $U \in \clU \cap \V_1 $. Thus $(\clU \cap \V_1 , \V \cap \clU_2)$ is a torsion pair in $\clH_{[t_1,t_2]}$.
		\par(2) holds by \Cref{lemma3.1}.
	\end{proof}
	
	\begin{Lemma}\label{lemma3.3}
		There exists a map 
		\begin{equation}\label{eq3.1}
			\begin{aligned}
				\Phi:\tors[t_1,t_2]&\rightarrow\tors\clH_{[t_1,t_2]}\\
				(\clU,\V)&\mapsto(\clU \cap \V_1 , \V \cap \clU_2)
			\end{aligned}
		\end{equation} 
		satisfying the following properties.
		\begin{enumerate}
			\item  By restriction, there is a map $\stors[t_1,t_2]\rightarrow\stors\clH_{[t_1,t_2]}$.
			\item $\Phi$ preserves $\preccurlyeq$. That is, for any  $a_1=(\X_1,\Y_1),\ a_2=(\X_2,\Y_2)\in \tors[t_1,t_2]$, if $a_1 \preccurlyeq a_2$, then $\Phi(a_1) \preccurlyeq \Phi(a_2)$.
			\item $\Phi$ preserves the hearts. That is, $\clH_{[a_1,a_2]}=\clH_{[\Phi(a_1),\Phi(a_2)]}$.
		\end{enumerate}
	\end{Lemma}
	\begin{proof} 
		By \Cref{lemma3.2}, $\Phi$ is well defined.
		\par (1) It is sufficient to verify (STP).  Indeed, it follows from $\mathbb{E}^{-1}(\clU \cap \V_1 , \V \cap \clU_2)\subseteq\mathbb{E}^{-1}(\clU,\V)=0$.
		\par (2) It is obvious that $\C(\X_1\cap\V_1,\Y_2\cap\clU_2)\subseteq \C(\X_1,\Y_2)=0$ and $\mathbb{E}^{-1}(\X_1\cap\V_1,\Y_2\cap\clU_2)\subseteq \mathbb{E}^{-1}(\X_1,\Y_2)=0$. Hence $\Phi(a_1) \preccurlyeq \Phi(a_2)$.
		\par (3) $\clH_{[\Phi(a_1),\Phi(a_2)]}=(\X_2\cap\V_1)\cap(\Y_1\cap\clU_2)=(\X_2\cap\clU_2)\cap(\Y_1\cap\V_1)=\X_2\cap\Y_1=\clH_{[a_1,a_2]}$.
	\end{proof}
	
	\begin{Lemma}\label{lemma3.4}
		Let $(\T,\F)$ be a torsion pair in $\clH_{[t_1,t_2]}$. Then the following holds.
		\begin{enumerate}
			\item $t=(\clU_1 * \T, \F * \V_2)$ is a torsion pair in $\C$.
			\item $(\clU_1 * \T) \cap \V_1 = \T$ and $(\F * \V_2) \cap \clU_2 = \F$.
			\item $t_1 \preccurlyeq t \preccurlyeq t_2$ if and only if $\mathbb{E}^{-1}(\T,\V_2)=0$ and $\mathbb{E}^{-1}(\clU_1,\F)=0$.
		\end{enumerate}
	\end{Lemma}
	\begin{proof}
		\par (1)  It is clear that $\C(\clU_1 * \T, \F * \V_2) = 0$. By \Cref{lemma3.1}, $\C = \clU_1 * \V_1 = \clU_1 * \clH_{[t_1,t_2]} * \V_2= (\clU_1*\T)*(\F*\V_2)$. Hence $(\clU_1 * \T, \F * \V_2)$ is a torsion pair in $\C$.
		\par (2) We only prove $(\clU_1*\T)\cap\V_1=\T$, the other equation is similar. Since $\T \subseteq \clH_{[t_1,t_2]} \subseteq \V_1$ and $\T \subseteq \clU_1*\T$, we have $\T \subseteq (\clU_1*\T)\cap\V_1$. Conversely, since $\clH_{[t_1,t_2]} ((\clU_1*\T)\cap \V_1,\F)\subseteq\C (\clU_1*\T,\F)$=0 by $\C(\clU_1,\F)\subseteq \C(\clU_1,\V_1)=0$ and $\C(\T,\F)=0$, we have  $(\clU_1*\T)\cap \V_1\subseteq \T$. Hence, $(\clU_1*\T)\cap \V_1= \T$.
		\par (3) Suppose $t_1 \preccurlyeq t$. By definition, we have $ \mathbb{E}^{-1}(\clU_1,\F*\V_2)=0$, hence $\mathbb{E}^{-1}(\clU_1,\F)=0$. Similarly, if $t \preccurlyeq t_2$, then $\mathbb{E}^{-1}(\T,\V_2)=0$.
		\par Suppose $\mathbb{E}^{-1}(\T,\V_2)=0$ and $\mathbb{E}^{-1}(\clU_1,\F)=0$. It is obvious that $\C(\clU_1,\F * \V_2) = 0$ and $\C( \clU_1 * \T, \V_2) = 0$. Since $\mathbb{E}^{-1}(\clU_1, \F) =0, \mathbb{E}^{-1}(\clU_1, \V_2) =0$ and $\mathbb{E}^{-1}(\T, \V_2) = 0$, we have $\mathbb{E}^{-1}(\clU_1,\F * \V_2)=0$ and $\mathbb{E}^{-1}( \clU_1 * \T, \V_2)=0$. Hence, $t_1 \preccurlyeq t \preccurlyeq t_2$.
	\end{proof}
	
	To establish a precise correspondence between the torsion pairs in $\C$ and the torsion pairs in $\clH_{[t_1,t_2]}$, we introduce the following notation.
	
	\begin{Not}
		Let  $t_1=(\clU_1,\V_1)$ and $t_2=(\clU_2,\V_2)$ be two torsion pairs in $\C$ with $t_1 \preccurlyeq t_2$ and $\clH_{[t_1,t_2]}=\clU_2\cap \V_1$. We define $$\widetilde{\tors}\clH_{[t_1,t_2]}=\{(\T,\F) \in \tors \clH_{[t_1,t_2]} \mid \mathbb{E}^{-1}(\T,\V_2)=0 \text{ and } \mathbb{E}^{-1}(\clU_1,\F)=0\}$$
		and $$\widetilde{\stors}\clH_{[t_1,t_2]}=\widetilde{\tors}\clH_{[t_1,t_2]} \cap\stors \clH_{[t_1,t_2]} .$$
	\end{Not}
	
	\begin{Remark}\label{re3.6}
		If $t_1$ and $t_2$ are $s$-torsion pairs, then for any torsion pair $(\T,\F)$ in $\clH_{[t_1,t_2]}$, we have $\mathbb{E}^{-1}(\T,\V_2)=0$ and $\mathbb{E}^{-1}(\clU_1,\F)=0$ since $\T \subseteq \clU_2$ and $\F \subseteq \V_1$. That is, if $t_1$ and $t_2$ are $s$-torsion pairs, then $\widetilde{\tors}\clH_{[t_1,t_2]}=\tors\clH_{[t_1,t_2]}$ and $\widetilde{\stors}\clH_{[t_1,t_2]}=\stors\clH_{[t_1,t_2]}$.
	\end{Remark}
	
	\begin{Lemma}\label{lemma3.7}
		There exists a map 
		\begin{equation}\label{eq3.2}
			\begin{aligned}
				\Psi:\widetilde{\tors}\clH_{[t_1,t_2]}&\rightarrow\tors[t_1,t_2]\\
				(\T,\F)&\mapsto(\clU_1 * \T , \F * \V_2)
			\end{aligned}
		\end{equation} 
		satisfying the following properties.
		\begin{enumerate}
			\item  By restriction, there is a map $\widetilde{\stors}\clH_{[t_1,t_2]}\rightarrow\stors[t_1,t_2]$.
			\item  $\Psi$ preserves $\preccurlyeq$. That is, for any  $b_1=(\T_1,\F_1),\ b_2=(\T_2,\F_2)\in \widetilde{\tors}\clH_{[t_1,t_2]}$, if $b_1 \preccurlyeq b_2$, then $\Psi(b_1) \preccurlyeq \Psi(b_2)$.
			\item $\Psi$ preserves the hearts. That is, $\clH_{[b_1,b_2]}=\clH_{[\Psi(b_1),\Psi(b_2)]}$.
		\end{enumerate}
	\end{Lemma}
	\begin{proof}
		
		By \Cref{lemma3.4},  $\Psi$ is well defined. 
		\par (1) It is sufficient to verify (STP). Since $\mathbb{E}^{-1}(\T,\F)=0$, $\mathbb{E}^{-1}(\clU_1,\V_2)=0$, $\mathbb{E}^{-1}(\T,\V_2)=0$ and $\mathbb{E}^{-1}(\clU_1,\F)=0$, we have $\mathbb{E}^{-1}(\clU_1 * \T, \F * \V_2)=0$.
		\par(2) Since $\clU_1*\T_1\subseteq \clU_1*\T_2$, it follows that $\C(\clU_1*\T_1,\F_2*\V_2)=0$. Since $\mathbb{E}^{-1}(\clU_1,\F_2)=0$, $\mathbb{E}^{-1}(\clU_1,\V_2)=0$, $\mathbb{E}^{-1}(\T_1,\V_2)=0$ and $\mathbb{E}^{-1}(\T_1,\F_2)=0$, it follows that $\mathbb{E}^{-1}(\clU_1*\T_1,\F_2*\V_2)=0$. Hence $\Psi(b_1)\preccurlyeq\Psi(b_2)$.
		\par(3) $\clH_{[\Psi(b_1),\Psi(b_2)]}=(\clU_1*\T_2)\cap(\F_1*\V_2)$. It follows from \Cref{lemma3.4} and \Cref{lemma3.1} that $(\clU_1*\T_2)\cap(\F_1*\V_2)\subseteq (\clU_1*\T_2)\cap\V_1=\T_2$ and $(\clU_1*\T_2)\cap(\F_1*\V_2)\subseteq \clU_2\cap(\F_1*\V_2)=\F_1$. Hence $(\clU_1*\T_2)\cap(\F_1*\V_2)\subseteq \T_2\cap \F_1$. It is clear that $(\clU_1*\T_2)\cap(\F_1*\V_2)\supseteq \T_2\cap \F_1$. Thus $\T_2\cap\F_1=(\clU_1*\T_2)\cap(\F_1*\V_2)$.
	\end{proof}
	
	Now we can establish our main result of this section.
	
	\begin{Thm}\label{theorem3.8}
		Let  $t_1=(\clU_1,\V_1)$ and $t_2=(\clU_2,\V_2)$ be two torsion pairs in $\C$ with $t_1 \preccurlyeq t_2$ and $\clH_{[t_1,t_2]}=\clU_2\cap \V_1$. Then there are order preserving, mutually inverse bijections
		$$\xymatrix@R=15pt{\tors[t_1,t_2] \ar@<0.5ex>[rr]^{\Phi\,\ } && 
			\widetilde{\tors}\clH_{[t_1,t_2]}, \ar@<0.5ex>[ll]^{\Psi\,\ }	}$$
		$$\xymatrix@R=15pt{\stors[t_1,t_2] \ar@<0.5ex>[rr]^{\Phi\,\ } && 
			\widetilde{\stors}\clH_{[t_1,t_2]} \ar@<0.5ex>[ll]^{\Psi\,\ }	}$$
		given by 
		$$	\Phi(\clU,\V)=(\clU \cap \V_1,\V \cap \clU_2),\ \Psi(\T,\F)=(\clU_1*\T,\F*\V_2).$$
	\end{Thm}
	\begin{proof}
		It follows from \Cref{lemma3.2}, \Cref{lemma3.3}, \Cref{lemma3.4} and \Cref{lemma3.7}.
	\end{proof}	
	
	In particular, if  $t_1$ and $t_2$ are two $s$-torsion pairs, by \Cref{re3.6} we have the following corollary, which  generalizes \cite[Theorem 3.9]{AET23}.
	
	\begin{Coro}\label{corollary3.9}
		Let $t_1=(\clU_1,\V_1)$ and $t_2=(\clU_2,\V_2)$ be two $s$-torsion pairs such that $\clU_1 \subseteq \clU_2$. Then there are order preserving, mutually inverse bijections
		$$\xymatrix@R=15pt{\tors[t_1,t_2] \ar@<0.5ex>[rr]^{\Phi\,\ } && \tors \clH_{[t_1,t_2]} \ar@<0.5ex>[ll]^{\Psi\,\ }	} \textrm{ and }\xymatrix@R=15pt{\stors[t_1,t_2] \ar@<0.5ex>[rr]^{\Phi\,\ } && \stors \clH_{[t_1,t_2]}, \ar@<0.5ex>[ll]^{\Psi\,\ }	}$$
		where the maps are given by \eqref{eq3.1} and \eqref{eq3.2}.
	\end{Coro}
	
	We now apply the above results to triangulated categories. Let $\D$ be a triangulated category. In this context, the relation $\preccurlyeq$ can be equivalently described as follows. Let  $t_1=(\clU_1,\V_1)$ and $t_2=(\clU_2,\V_2)$ be two torsion pairs in $\D$, then $t_1 \preccurlyeq t_2$ if and only if $\clU_1 \subseteq \clU_2$ and $\clU_1[1] \subseteq \clU_2$.
	
	\begin{Coro}\label{corollary3.10}	
		Let $\D$ be a triangulated category and $t_1=(\clU_1,\V_1)$ and $t_2=(\clU_2,\V_2)$ be two torsion pairs in $\D$ such that $\clU_1 \subseteq \clU_2$ and $\clU_1[1] \subseteq \clU_2$. Let $\clH_{[t_1,t_2]}=\clU_2\cap \V_1$. Denote by $\tstr[t_1,t_2]$ the set of $t$-structures $(\D^{\leqslant 0},\D^{\geqslant 0})$ on $\D$ such that $\clU_1 \subseteq \D^{\leqslant 0}\subseteq \clU_2$ and by $\widetilde{\tors}\clH_{[t_1,t_2]}$ (resp. $\widetilde{\stors}\clH_{[t_1,t_2]}$) the set of torsion pairs (resp. $s$-torsion pairs) $(\T,\F)$ in $\clH_{[t_1,t_2]}$ such that $\T[1] \subseteq \clU_2$ and $\F[-1]\subseteq \V_1$. Then there are the following order preserving, mutually inverse bijections.
		\begin{enumerate}
			\item $\tors[t_1,t_2] \overset{\Phi}{\underset{\Psi}{\rightleftarrows}}  \widetilde{\tors}\clH_{[t_1,t_2]}$ where $\Phi$ and $\Psi$ are given in \eqref{eq3.1} and \eqref{eq3.2}.
			\item  $\tstr[t_1,t_2] \overset{\phi}{\underset{\psi}{\rightleftarrows}} \widetilde{\stors}\clH_{[t_1,t_2]}$ where $\phi(\D^{\leqslant 0},\D^{\geqslant 0})=(\D^{\leqslant 0}\cap\V_1,\D^{\geqslant 1}\cap\clU_2)$ and $\psi(\T,\F)=(\clU_1*\T,\F[1]*\V_2[1])$.
		\end{enumerate}
	\end{Coro}
	
	In particular, assume that $t_1=(\clU_1,\V_1)$ and $t_2=(\clU_2,\V_2)$ are $s$-torsion pairs, that is, $s_1=(\clU_1,\V_1[1])$ and $s_2=(\clU_2,\V_2[1])$ are t-structures. Define $\tors[s_1,s_2]=\tors[t_1,t_2]$. We have the following corollary, which generalizes \cite[Corollary 3.14]{AET23}.
	
	\begin{Coro}\label{corollary3.11}
		Let $\D$ be a triangulated category with two $t$-structures  $s_1=(\C_1^{\leqslant 0},\C_1^{\geqslant 0})$ and  $s_2=(\C_2^{\leqslant 0},\C_2^{\geqslant 0})$ such that $\C_1^{\leqslant 0}\subseteq \C_2^{\leqslant 0}$. Let $\clH_{[s_1,s_2]}=\C_2^{\leqslant 0}\cap \C_1^{\geqslant 1}$. Denote by $\tstr[s_1,s_2]$ the set of $t$-structures $(\D^{\leqslant 0},\D^{\geqslant 0})$ on $\D$ such that $\C_1^{\leqslant 0}\subseteq \D^{\leqslant 0}\subseteq \C_2^{\leqslant 0}$. Then there are the following order preserving, mutually inverse bijections.
		\begin{enumerate}
			\item $\tors[s_1,s_2] \overset{\phi}{\underset{\psi}{\rightleftarrows}} \tors\clH_{[s_1,s_2]}$ where $\phi(\clU,\V)=(\clU\cap\C_1^{\geqslant 1},\V\cap \C_2^{\leqslant 0})$ and $\psi(\T,\F)=(\C_1^{\leqslant 0}*\T,\F*\C_2^{\geqslant 1})$.
			\item $\tstr[s_1,s_2] \overset{\phi}{\underset{\psi}{\rightleftarrows}} \stors\clH_{[s_1,s_2]}$ where $\phi(\D^{\leqslant 0},\D^{\geqslant 0})=(\Dl0\cap\C_1^{\geqslant 1},\Dg1\cap \C_2^{\leqslant 0})$ and $\psi(\T,\F)=(\C_1^{\leqslant 0}*\T,\F[1]*\C_2^{\geqslant 0})$.
		\end{enumerate}
	\end{Coro}
	
	We finish this section by giving two concrete examples.
	
	\begin{Ex}
		Let $Q$ be the quiver $1 \rightarrow 2 \leftarrow 3 \leftarrow 4$. Then the path algebra $kQ$ is a hereditary algebra and the module category $\mod kQ$  is an extriangulated category with the negative first extension $\mathbb{E} ^{-1} (-,-) = \operatorname{Ext} ^1 _{kQ} (-,-)$ by \Cref{example2.3} (3). The Auslander-Reiten quiver of $\mod kQ$ is in \Cref{figure1}. 
		
		\begin{figure}[htp]
			\begin{tikzpicture}[scale = 1]
				\node (3) at (3,2) {$3$};
				\node (32) at (1,0) {$\substack{3 \\ 2}$};
				\node (123) at (2,1) {$\substack{1\,3 \\ 2}$};
				\node (2) at (0,1) {$2$};
				\node (12) at (1,2) {$\substack{1 \\ 2}$};
				\node (1432) at (3,0) {$\substack{\mathrm{~} \, 4\\ 1\, 3 \\ 2}$};
				\node (4) at (5,2) {$4$};
				\node (43) at (4,1) {$\substack{4 \\ 3}$};
				\node (432) at (2,-1) {$\substack{4 \\ 3 \\ 2}$};
				\node (1) at (4,-1) {$1$};
				\draw[->] (2) -- (12);
				\draw[->] (2) -- (32);
				\draw[->] (32) -- (123);
				\draw[->] (12) -- (123);
				\draw[->] (123) -- (3);
				\draw[->] (123) -- (1432);
				\draw[->] (32) -- (432);
				\draw[->] (432) -- (1432);
				\draw[->] (1432) -- (1);
				\draw[->] (3) -- (43);
				\draw[->] (1432) -- (43);
				\draw[->] (43) -- (4);
			\end{tikzpicture}
			\caption{The Auslander-Reiten quiver of $\mod kQ$}\label{figure1}
		\end{figure}
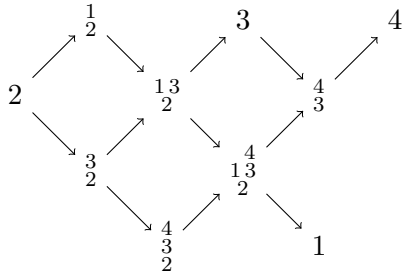
		
		For an object $M$ in $\mod kQ$, we denote by $\add M$ the smallest additive full subcategory of $\mod kQ$ which contains $M$ and is closed under taking finite direct sums and direct summands. Let $\clU_1=\add {2}$, $\V_1=\add {3 \oplus 1 \oplus \substack{4 \\ 3} \oplus 4}$, $\clU_2=\add {2 \oplus \substack{3 \\ 2} \oplus \substack{4 \\ 3 \\ 2} \oplus 3 \oplus \substack{4 \\ 3} \oplus 4}$ and $\V_2=\add {1}$. Then one can verify that  $t_1=(\clU_1,\V_1)$ and $t_2=(\clU_2,\V_2)$ are two $s$-torsion pairs with $\clU_1 \subseteq \clU_2$, and $\clH_{[t_1,t_2]}=\V_1\cap\clU_2=\add{3 \oplus \substack{4 \\ 3} \oplus 4}$. By computation, we obtain that there are exactly five torsion pairs in both $\tors[t_1,t_2]$ and $\tors\clH_{[t_1,t_2]}$. The Hasse quivers of $\tors [t_1,t_2]$ and $\tors\clH_{[t_1,t_2]}$ are shown in \Cref{figure2}, where the indecomposable objects in the torsion classes are marked in black, and those in the torsionfree classes are marked in white. Furthermore, there are three elements in each of $\stors[t_1,t_2]$ and $\stors\clH_{[t_1,t_2]}$, which are enclosed in boxes in  \Cref{figure2}. 
		
		\begin{figure}[htp]
			\begin{picture}(400,210)(0,0)
				\put(0,0){
					\begin{tikzpicture}[yscale=-0.9] 
						\node (a) at (100pt,0pt){
							\begin{tikzpicture}[scale=0.35, every node/.style={scale=0.5}]
								\node (2) at (0,2) [black] {};
								\node (12) at (1,3) [] {};
								\node (23) at (1,1) [black] {};
								\node (123) at (2,2) [] {};
								\node (234) at (2,0) [black] {};
								\node (3) at (3,3) [black] {};
								\node (1234) at (3,1) [] {};
								\node (34) at (4,2) [black] {};
								\node (1) at (4,0) [white] {};
								\node (4) at (5,3) [black] {};
								\draw[->] (2) -- (12);
								\draw[->] (2) -- (23);
								\draw[->] (12) -- (123);
								\draw[->] (23) -- (123);
								\draw[->] (23) -- (234);
								\draw[->] (123) -- (3);
								\draw[->] (123) -- (1234);
								\draw[->] (234) -- (1234);
								\draw[->] (3) -- (34);
								\draw[->] (1234) -- (34);
								\draw[->] (1234) -- (1);
								\draw[->] (34) -- (4);
								\node at (0,-.2) {};
								\node at (0,3.2) {};
								\draw[thick] ([shift={(-0.3,-0.3)}]current bounding box.south west) rectangle ([shift={(0.3,0.3)}] current bounding box.north east);
						\end{tikzpicture}} ;
						\node (b) at (20pt,60pt){
							\begin{tikzpicture}[scale=0.35, every node/.style={scale=0.5}]
								\node (2) at (0,2) [black] {};
								\node (12) at (1,3) [] {};
								\node (23) at (1,1) [] {};
								\node (123) at (2,2) [] {};
								\node (234) at (2,0) [black] {};
								\node (3) at (3,3) [white] {};
								\node (1234) at (3,1) [] {};
								\node (34) at (4,2) [black] {};
								\node (1) at (4,0) [white] {};
								\node (4) at (5,3) [black] {};
								\draw[->] (2) -- (12);
								\draw[->] (2) -- (23);
								\draw[->] (12) -- (123);
								\draw[->] (23) -- (123);
								\draw[->] (23) -- (234);
								\draw[->] (123) -- (3);
								\draw[->] (123) -- (1234);
								\draw[->] (234) -- (1234);
								\draw[->] (3) -- (34);
								\draw[->] (1234) -- (34);
								\draw[->] (1234) -- (1);
								\draw[->] (34) -- (4);
								\node at (0,-.2) {};
								\node at (0,3.2) {};
						\end{tikzpicture}} ;
						\node (c) at (20pt,120pt){
							\begin{tikzpicture}[scale=0.35, every node/.style={scale=0.5}]
								\node (2) at (0,2) [black] {};
								\node (12) at (1,3) [] {};
								\node (23) at (1,1) [] {};
								\node (123) at (2,2) [] {};
								\node (234) at (2,0) [] {};
								\node (3) at (3,3) [white] {};
								\node (1234) at (3,1) [] {};
								\node (34) at (4,2) [white] {};
								\node (1) at (4,0) [white] {};
								\node (4) at (5,3) [black] {};
								\draw[->] (2) -- (12);
								\draw[->] (2) -- (23);
								\draw[->] (12) -- (123);
								\draw[->] (23) -- (123);
								\draw[->] (23) -- (234);
								\draw[->] (123) -- (3);
								\draw[->] (123) -- (1234);
								\draw[->] (234) -- (1234);
								\draw[->] (3) -- (34);
								\draw[->] (1234) -- (34);
								\draw[->] (1234) -- (1);
								\draw[->] (34) -- (4);
								\node at (0,-.2) {};
								\node at (0,3.2) {};
						\end{tikzpicture}} ;
						\node (d) at (180pt,90pt){
							\begin{tikzpicture}[scale=0.35, every node/.style={scale=0.5}]
								\node (2) at (0,2) [black] {};
								\node (12) at (1,3) [] {};
								\node (23) at (1,1) [black] {};
								\node (123) at (2,2) [] {};
								\node (234) at (2,0) [] {};
								\node (3) at (3,3) [black] {};
								\node (1234) at (3,1) [] {};
								\node (34) at (4,2) [] {};
								\node (1) at (4,0) [white] {};
								\node (4) at (5,3) [white] {};
								\draw[->] (2) -- (12);
								\draw[->] (2) -- (23);
								\draw[->] (12) -- (123);
								\draw[->] (23) -- (123);
								\draw[->] (23) -- (234);
								\draw[->] (123) -- (3);
								\draw[->] (123) -- (1234);
								\draw[->] (234) -- (1234);
								\draw[->] (3) -- (34);
								\draw[->] (1234) -- (34);
								\draw[->] (1234) -- (1);
								\draw[->] (34) -- (4);
								\node at (0,-.2) {};
								\node at (0,3.2) {};
								\draw[thick] ([shift={(-0.3,-0.3)}]current bounding box.south west) rectangle ([shift={(0.3,0.3)}]current bounding box.north east);
						\end{tikzpicture}} ;
						\node (e) at (100pt,180pt){
							\begin{tikzpicture}[scale=0.35, every node/.style={scale=0.5}]
								\node (2) at (0,2) [black] {};
								\node (12) at (1,3) [] {};
								\node (23) at (1,1) [] {};
								\node (123) at (2,2) [] {};
								\node (234) at (2,0) [] {};
								\node (3) at (3,3) [white] {};
								\node (1234) at (3,1) [] {};
								\node (34) at (4,2) [white] {};
								\node (1) at (4,0) [white] {};
								\node (4) at (5,3) [white] {};
								\draw[->] (2) -- (12);
								\draw[->] (2) -- (23);
								\draw[->] (12) -- (123);
								\draw[->] (23) -- (123);
								\draw[->] (23) -- (234);
								\draw[->] (123) -- (3);
								\draw[->] (123) -- (1234);
								\draw[->] (234) -- (1234);
								\draw[->] (3) -- (34);
								\draw[->] (1234) -- (34);
								\draw[->] (1234) -- (1);
								\draw[->] (34) -- (4);
								\node at (0,-.2) {};
								\node at (0,3.2) {};
								\draw[thick] ([shift={(-0.3,-0.3)}]current bounding box.south west) rectangle ([shift={(0.3,0.3)}]current bounding box.north east);
						\end{tikzpicture}} ;
						\draw[->, shorten >=4pt, shorten <=4pt] (a)--(b);
						\draw[->, shorten >=4pt, shorten <=4pt] (a)--(d);
						\draw[->, shorten >=2pt, shorten <=2pt] (b)--(c);
						\draw[->, shorten >=4pt, shorten <=4pt] (c)--(e);
						\draw[->, shorten >=4pt, shorten <=4pt] (d)--(e);
				\end{tikzpicture}}
				\put(250,30){
					\begin{tikzpicture}[yscale=-0.9] 
						\node (a) at (60pt,0pt){
							\begin{tikzpicture}[scale=0.35, every node/.style={scale=0.5}]
								\node (3) at (3,3) [black] {};
								\node (34) at (4,2) [black] {};
								\node (4) at (5,3) [black] {};
								\draw[->] (3) -- (34);
								\draw[->] (34) -- (4);
								\draw[thick] ([shift={(-0.3,-0.3)}]current bounding box.south west) rectangle ([shift={(0.3,0.3)}]current bounding box.north east);
						\end{tikzpicture}} ;
						\node (b) at (0pt,45pt){
							\begin{tikzpicture}[scale=0.35, every node/.style={scale=0.5}]
								\node (3) at (3,3) [white] {};
								\node (34) at (4,2) [black] {};
								\node (4) at (5,3) [black] {};
								\draw[->] (3) -- (34);
								\draw[->] (34) -- (4);
						\end{tikzpicture}} ;
						\node (c) at (0pt,90pt){
							\begin{tikzpicture}[scale=0.35, every node/.style={scale=0.5}]
								\node (3) at (3,3) [white] {};
								\node (34) at (4,2) [white] {};
								\node (4) at (5,3) [black] {};
								\draw[->] (3) -- (34);
								\draw[->] (34) -- (4);
						\end{tikzpicture}} ;
						\node (d) at (120pt,67.5pt){
							\begin{tikzpicture}[scale=0.35, every node/.style={scale=0.5}]
								\node (3) at (3,3) [black] {};
								\node (34) at (4,2) [] {};
								\node (4) at (5,3) [white] {};
								\draw[->] (3) -- (34);
								\draw[->] (34) -- (4);
								\draw[thick] ([shift={(-0.3,-0.3)}]current bounding box.south west) rectangle ([shift={(0.3,0.3)}]current bounding box.north east);
						\end{tikzpicture}} ;
						\node (e) at (60pt,135pt){
							\begin{tikzpicture}[scale=0.35, every node/.style={scale=0.5}]
								\node (3) at (3,3) [white] {};
								\node (34) at (4,2) [white] {};
								\node (4) at (5,3) [white] {};
								\draw[->] (3) -- (34);
								\draw[->] (34) -- (4);
								\draw[thick] ([shift={(-0.3,-0.3)}]current bounding box.south west) rectangle ([shift={(0.3,0.3)}]current bounding box.north east);
						\end{tikzpicture}} ;
						\draw[->, shorten >=4pt, shorten <=4pt] (a)--(b);
						\draw[->, shorten >=4pt, shorten <=4pt] (a)--(d);
						\draw[->, shorten >=2pt, shorten <=2pt] (b)--(c);
						\draw[->, shorten >=4pt, shorten <=4pt] (c)--(e);
						\draw[->, shorten >=4pt, shorten <=4pt] (d)--(e);
				\end{tikzpicture}}
			\end{picture}
			\caption{The Hasse quivers of $\tors [t_1,t_2]$ and $\tors\clH_{[t_1,t_2]}$}\label{figure2}
		\end{figure}
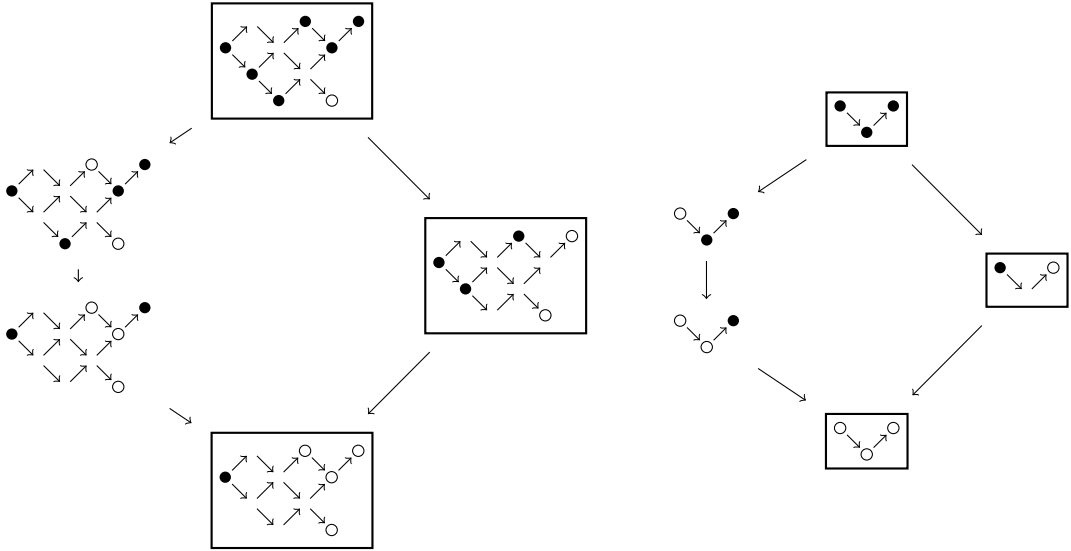
	\end{Ex}

	\begin{Ex}\label{example3.13}
		Let $Q$ be the Dynkin quiver  $1\rightarrow 2$. The bounded derived category $\D=\Db {\mod kQ}$ is a triangulated category with shift functor $[1]$ and the negative first extension functor $\mathbb{E} ^{-1} (-,-) = \D(-,-[-1])$. The Auslander-Reiten quiver of $\D$ is shown in \Cref{figure3}.
		
		\begin{figure}[htp]
			\begin{tikzpicture}[scale = 0.68]
				\node (0) at (-1,1) {$\cdots$};
				\node (1-3) at (0,2) {$2[\text{-}3]$};
				\node (2-3) at (1,0) {$\substack{1\\ 2}[\text{-}3]$};
				\node (3-3) at (2,2) {$1[\text{-}3]$};
				\node (1-2) at (3,0) {$2[\text{-}2]$};
				\node (2-2) at (4,2) {$\substack{1\\ 2}[\text{-}2]$};
				\node (3-2) at (5,0) {$1[\text{-}2]$};
				\node (1-1) at (6,2) {$2[\text{-}1]$};
				\node (2-1) at (7,0) {$\substack{1\\ 2}[\text{-}1]$};
				\node (3-1) at (8,2) {$1[\text{-}1]$};
				\node (1) at (9,0) {$2$};
				\node (2) at (10,2) {$\substack{1\\ 2}$};
				\node (3) at (11,0) {$1$};
				\node (11) at (12,2) {$2[1]$};
				\node (21) at (13,0) {$\substack{1\\ 2}[1]$};
				\node (31) at (14,2) {$1[1]$};
				\node (12) at (15,0) {$2[2]$};
				\node (22) at (16,2) {$\substack{1\\ 2}[2]$};
				\node (32) at (17,0) {$1[2]$};
				\node (99) at (18,1) {$\cdots$};
				\draw[->] (1-3) -- (2-3);
				\draw[->] (2-3) -- (3-3);
				\draw[->] (3-3) -- (1-2);
				\draw[->] (1-2) -- (2-2);
				\draw[->] (2-2) -- (3-2);
				\draw[->] (3-2) -- (1-1);
				\draw[->] (1-1) -- (2-1);
				\draw[->] (2-1) -- (3-1);
				\draw[->] (3-1) -- (1);
				\draw[->] (1) -- (2);
				\draw[->] (2) -- (3);
				\draw[->] (3) -- (11);
				\draw[->] (11) -- (21);
				\draw[->] (21) -- (31);
				\draw[->] (31) -- (12);
				\draw[->] (12) -- (22);
				\draw[->] (22) -- (32);
			\end{tikzpicture}
			\caption{The Auslander-Reiten quiver of $D^b(\mod kQ)$}\label{figure3}
		\end{figure}
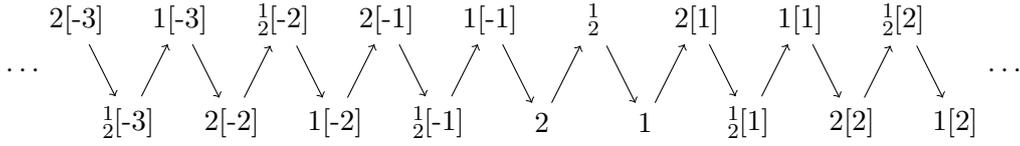
		
		Consider two torsion pairs  $t_1=(\clU_1,\V_1)$ and $t_2=(\clU_2,\V_2)$ in $\D$, as depicted in \Cref{figure4}. Here, black nodes represent objects belonging to the torsion classes, white nodes represent those in the torsionfree classes, and square nodes belong to $\mod kQ$. It is straightforward to verify that  $ t_1 \preccurlyeq t_2 $, that is,  $\clU_1 \subseteq \clU_2$ and $\clU_1[1] \subseteq \clU_2$.
		
		A direct computation shows that there are exactly four torsion pairs belonging to $\tors[t_1,t_2]$, denoted by $a_i=(\X_i,\Y_i)$, $i=1,2,3,4$, among which $a_1$, $a_3$ and $a_4$ are $s$-torsion pairs and enclosed in boxes in \Cref{figure4}.
		
		The heart is given by 
		$$\clH_{[t_1,t_2]}=\clU_2\cap\V_1=\add {2[\text{-}2] \oplus \substack{1\\ 2}[\text{-}2] \oplus 1[\text{-}2] \oplus 2[\text{-}1] \oplus 2[1]}.$$
		Within this heart, the set $\widetilde{\tors}\clH_{[t_1,t_2]}$ of torsion pairs $( \T,\F)$ satisfying  $\T[1] \subseteq \clU_2$ and $\F[-1]\subseteq \V_1$ also contains exactly four elements, denoted by $b_i=(\T_i,\F_i)$, $i=1,2,3,4$. Among these,  $b_1$, $b_3$ and $b_4$ are $s$-torsion pairs and enclosed in boxes in \Cref{figure4}.
		
		\begin{figure}[htp]
			\begin{picture}(350,150)(0,0)
				\put(0,75){
					\begin{tabular}{c|c}
						& {$\tors \D$}\\
						\hline \hline
						\\[-16pt] {$t_1$} & 
						\begin{tikzpicture}[baseline={([yshift=-.5ex]current bounding box.center)}, scale=0.35,  every node/.style={scale=0.5}];
							\node (0) at (-1,0.5) {$\cdots$};
							\node[white] (1-3) at (0,1) {};
							\node[white] (2-3) at (1,0) {};
							\node[white] (3-3) at (2,1) {};
							\node[white] (1-2) at (3,0) {};
							\node[white] (2-2) at (4,1) {};
							\node[white] (3-2) at (5,0) {};
							\node[white] (1-1) at (6,1) {};
							\node[white] (2-1) at (7,0) {};
							\node[white] (3-1) at (8,1) {};
							\node[b] (1) at (9,0) {};
							\node[b] (2) at (10,1) {};
							\node (3) at (11,0) {};
							\node[white] (11) at (12,1) {};
							\node[black] (21) at (13,0) {};
							\node[black] (31) at (14,1) {};
							\node[black] (12) at (15,0) {};
							\node[black] (22) at (16,1) {};
							\node[black] (32) at (17,0) {};
							\node (99) at (18,0.5) {$\cdots$};
							\draw[->] (1-3) -- (2-3);
							\draw[->] (2-3) -- (3-3);
							\draw[->] (3-3) -- (1-2);
							\draw[->] (1-2) -- (2-2);
							\draw[->] (2-2) -- (3-2);
							\draw[->] (3-2) -- (1-1);
							\draw[->] (1-1) -- (2-1);
							\draw[->] (2-1) -- (3-1);
							\draw[->] (3-1) -- (1);
							\draw[->] (1) -- (2);
							\draw[->] (2) -- (3);
							\draw[->] (3) -- (11);
							\draw[->] (11) -- (21);
							\draw[->] (21) -- (31);
							\draw[->] (31) -- (12);
							\draw[->] (12) -- (22);
							\draw[->] (22) -- (32);
						\end{tikzpicture} \\ [2pt]
						\hline\\[-16pt]
						{$\boxed{\smash{a_1}\vphantom{a}}$}&
						\begin{tikzpicture}[baseline={([yshift=-.5ex]current bounding box.center)}, scale=0.35,  every node/.style={scale=0.5}];
							\node (0) at (-1,0.5) {$\cdots$};
							\node[white] (1-3) at (0,1) {};
							\node[white] (2-3) at (1,0) {};
							\node[white] (3-3) at (2,1) {};
							\node[white] (1-2) at (3,0) {};
							\node[white] (2-2) at (4,1) {};
							\node[white] (3-2) at (5,0) {};
							\node[white] (1-1) at (6,1) {};
							\node[white] (2-1) at (7,0) {};
							\node[white] (3-1) at (8,1) {};
							\node[b] (1) at (9,0) {};
							\node[b] (2) at (10,1) {};
							\node[b] (3) at (11,0) {};
							\node[black] (11) at (12,1) {};
							\node[black] (21) at (13,0) {};
							\node[black] (31) at (14,1) {};
							\node[black] (12) at (15,0) {};
							\node[black] (22) at (16,1) {};
							\node[black] (32) at (17,0) {};
							\node (99) at (18,0.5) {$\cdots$};
							\draw[->] (1-3) -- (2-3);
							\draw[->] (2-3) -- (3-3);
							\draw[->] (3-3) -- (1-2);
							\draw[->] (1-2) -- (2-2);
							\draw[->] (2-2) -- (3-2);
							\draw[->] (3-2) -- (1-1);
							\draw[->] (1-1) -- (2-1);
							\draw[->] (2-1) -- (3-1);
							\draw[->] (3-1) -- (1);
							\draw[->] (1) -- (2);
							\draw[->] (2) -- (3);
							\draw[->] (3) -- (11);
							\draw[->] (11) -- (21);
							\draw[->] (21) -- (31);
							\draw[->] (31) -- (12);
							\draw[->] (12) -- (22);
							\draw[->] (22) -- (32);
						\end{tikzpicture} \\ [2pt]
						\hline\\[-16pt]
						{$a_2$}&
						\begin{tikzpicture}[baseline={([yshift=-.5ex]current bounding box.center)}, scale=0.35,  every node/.style={scale=0.5}];
							\node (0) at (-1,0.5) {$\cdots$};
							\node[white] (1-3) at (0,1) {};
							\node[white] (2-3) at (1,0) {};
							\node[white] (3-3) at (2,1) {};
							\node[black] (1-2) at (3,0) {};
							\node (2-2) at (4,1) {};
							\node[white] (3-2) at (5,0) {};
							\node[white] (1-1) at (6,1) {};
							\node[white] (2-1) at (7,0) {};
							\node[white] (3-1) at (8,1) {};
							\node[b] (1) at (9,0) {};
							\node[b] (2) at (10,1) {};
							\node[b] (3) at (11,0) {};
							\node[black] (11) at (12,1) {};
							\node[black] (21) at (13,0) {};
							\node[black] (31) at (14,1) {};
							\node[black] (12) at (15,0) {};
							\node[black] (22) at (16,1) {};
							\node[black] (32) at (17,0) {};
							\node (99) at (18,0.5) {$\cdots$};
							\draw[->] (1-3) -- (2-3);
							\draw[->] (2-3) -- (3-3);
							\draw[->] (3-3) -- (1-2);
							\draw[->] (1-2) -- (2-2);
							\draw[->] (2-2) -- (3-2);
							\draw[->] (3-2) -- (1-1);
							\draw[->] (1-1) -- (2-1);
							\draw[->] (2-1) -- (3-1);
							\draw[->] (3-1) -- (1);
							\draw[->] (1) -- (2);
							\draw[->] (2) -- (3);
							\draw[->] (3) -- (11);
							\draw[->] (11) -- (21);
							\draw[->] (21) -- (31);
							\draw[->] (31) -- (12);
							\draw[->] (12) -- (22);
							\draw[->] (22) -- (32);
						\end{tikzpicture} \\ [2pt]
						\hline\\[-16pt]
						{$\boxed{\smash{a_3}\vphantom{a}}$}&
						\begin{tikzpicture}[baseline={([yshift=-.5ex]current bounding box.center)}, scale=0.35,  every node/.style={scale=0.5}];
							\node (0) at (-1,0.5) {$\cdots$};
							\node[white] (1-3) at (0,1) {};
							\node[white] (2-3) at (1,0) {};
							\node[white] (3-3) at (2,1) {};
							\node[white] (1-2) at (3,0) {};
							\node[white] (2-2) at (4,1) {};
							\node[white] (3-2) at (5,0) {};
							\node[black] (1-1) at (6,1) {};
							\node (2-1) at (7,0) {};
							\node[white] (3-1) at (8,1) {};
							\node[b] (1) at (9,0) {};
							\node[b] (2) at (10,1) {};
							\node[b] (3) at (11,0) {};
							\node[black] (11) at (12,1) {};
							\node[black] (21) at (13,0) {};
							\node[black] (31) at (14,1) {};
							\node[black] (12) at (15,0) {};
							\node[black] (22) at (16,1) {};
							\node[black] (32) at (17,0) {};
							\node (99) at (18,0.5) {$\cdots$};
							\draw[->] (1-3) -- (2-3);
							\draw[->] (2-3) -- (3-3);
							\draw[->] (3-3) -- (1-2);
							\draw[->] (1-2) -- (2-2);
							\draw[->] (2-2) -- (3-2);
							\draw[->] (3-2) -- (1-1);
							\draw[->] (1-1) -- (2-1);
							\draw[->] (2-1) -- (3-1);
							\draw[->] (3-1) -- (1);
							\draw[->] (1) -- (2);
							\draw[->] (2) -- (3);
							\draw[->] (3) -- (11);
							\draw[->] (11) -- (21);
							\draw[->] (21) -- (31);
							\draw[->] (31) -- (12);
							\draw[->] (12) -- (22);
							\draw[->] (22) -- (32);
						\end{tikzpicture} \\ [2pt]
						\hline\\[-16pt]
						{$\boxed{\smash{a_4}\vphantom{a}}$}&
						\begin{tikzpicture}[baseline={([yshift=-.5ex]current bounding box.center)}, scale=0.35,  every node/.style={scale=0.5}];
							\node (0) at (-1,0.5) {$\cdots$};
							\node[white] (1-3) at (0,1) {};
							\node[white] (2-3) at (1,0) {};
							\node[white] (3-3) at (2,1) {};
							\node[black] (1-2) at (3,0) {};
							\node (2-2) at (4,1) {};
							\node[white] (3-2) at (5,0) {};
							\node[black] (1-1) at (6,1) {};
							\node (2-1) at (7,0) {};
							\node[white] (3-1) at (8,1) {};
							\node[b] (1) at (9,0) {};
							\node[b] (2) at (10,1) {};
							\node[b] (3) at (11,0) {};
							\node[black] (11) at (12,1) {};
							\node[black] (21) at (13,0) {};
							\node[black] (31) at (14,1) {};
							\node[black] (12) at (15,0) {};
							\node[black] (22) at (16,1) {};
							\node[black] (32) at (17,0) {};
							\node (99) at (18,0.5) {$\cdots$};
							\draw[->] (1-3) -- (2-3);
							\draw[->] (2-3) -- (3-3);
							\draw[->] (3-3) -- (1-2);
							\draw[->] (1-2) -- (2-2);
							\draw[->] (2-2) -- (3-2);
							\draw[->] (3-2) -- (1-1);
							\draw[->] (1-1) -- (2-1);
							\draw[->] (2-1) -- (3-1);
							\draw[->] (3-1) -- (1);
							\draw[->] (1) -- (2);
							\draw[->] (2) -- (3);
							\draw[->] (3) -- (11);
							\draw[->] (11) -- (21);
							\draw[->] (21) -- (31);
							\draw[->] (31) -- (12);
							\draw[->] (12) -- (22);
							\draw[->] (22) -- (32);
						\end{tikzpicture} \\ [2pt] 
						\hline	\\[-16pt]	
						{$t_2$}&
						\begin{tikzpicture}[baseline={([yshift=-.5ex]current bounding box.center)}, scale=0.35,  every node/.style={scale=0.5}];
							\node (0) at (-1,0.5) {$\cdots$};
							\node[white] (1-3) at (0,1) {};
							\node[white] (2-3) at (1,0) {};
							\node[white] (3-3) at (2,1) {};
							\node[black] (1-2) at (3,0) {};
							\node[black] (2-2) at (4,1) {};
							\node[black] (3-2) at (5,0) {};
							\node[black] (1-1) at (6,1) {};
							\node (2-1) at (7,0) {};
							\node[white] (3-1) at (8,1) {};
							\node[b] (1) at (9,0) {};
							\node[b] (2) at (10,1) {};
							\node[b] (3) at (11,0) {};
							\node[black] (11) at (12,1) {};
							\node[black] (21) at (13,0) {};
							\node[black] (31) at (14,1) {};
							\node[black] (12) at (15,0) {};
							\node[black] (22) at (16,1) {};
							\node[black] (32) at (17,0) {};
							\node (99) at (18,0.5) {$\cdots$};
							\draw[->] (1-3) -- (2-3);
							\draw[->] (2-3) -- (3-3);
							\draw[->] (3-3) -- (1-2);
							\draw[->] (1-2) -- (2-2);
							\draw[->] (2-2) -- (3-2);
							\draw[->] (3-2) -- (1-1);
							\draw[->] (1-1) -- (2-1);
							\draw[->] (2-1) -- (3-1);
							\draw[->] (3-1) -- (1);
							\draw[->] (1) -- (2);
							\draw[->] (2) -- (3);
							\draw[->] (3) -- (11);
							\draw[->] (11) -- (21);
							\draw[->] (21) -- (31);
							\draw[->] (31) -- (12);
							\draw[->] (12) -- (22);
							\draw[->] (22) -- (32);
						\end{tikzpicture} \\ [2pt]
						\hline			
				\end{tabular}}
				\put(250,75){
					\begin{tabular}{c|c}
						& {$\widetilde{\tors}\clH_{[t_1,t_2]}$} \\ 
						\hline \hline \\[-16pt]
						{$\boxed{\smash{b_1}\vphantom{a}}$}&
						\begin{tikzpicture}[baseline={([yshift=-.5ex]current bounding box.center)}, scale=0.35,  every node/.style={scale=0.5}]
							\node[white] (1-2) at (3,0) {};
							\node[white] (2-2) at (4,1) {};
							\node[white] (3-2) at (5,0) {};
							\node[white] (1-1) at (6,1) {};
							\node (7) at (6.7,0) {};
							\node (8) at (7.4,1) {};
							\node (9) at (8.1,0) {};
							\node (10) at (8.8,1) {};
							\node (11) at (9.5,0) {};
							\node[black] (12) at (10.2,1) {};
							\draw[->] (1-2) -- (2-2);
							\draw[->] (2-2) -- (3-2);
							\draw[->] (3-2) -- (1-1);
							\draw[->] (1-1) -- (7);
							\draw[->] (7) -- (8);
							\draw[->] (8) -- (9);
							\draw[->] (9) -- (10);
							\draw[->] (10) -- (11);
							\draw[->] (11) -- (12);
						\end{tikzpicture} \\ [2pt]
						\hline \\[-16pt]
						{$b_2$}&
						\begin{tikzpicture}[baseline={([yshift=-.5ex]current bounding box.center)}, scale=0.35,  every node/.style={scale=0.5}]
							\node[black] (1-2) at (3,0) {};
							\node (2-2) at (4,1) {};
							\node[white] (3-2) at (5,0) {};
							\node[white] (1-1) at (6,1) {};
							\node (7) at (6.7,0) {};
							\node (8) at (7.4,1) {};
							\node (9) at (8.1,0) {};
							\node (10) at (8.8,1) {};
							\node (11) at (9.5,0) {};
							\node[black] (12) at (10.2,1) {};
							\draw[->] (1-2) -- (2-2);
							\draw[->] (2-2) -- (3-2);
							\draw[->] (3-2) -- (1-1);
							\draw[->] (1-1) -- (7);
							\draw[->] (7) -- (8);
							\draw[->] (8) -- (9);
							\draw[->] (9) -- (10);
							\draw[->] (10) -- (11);
							\draw[->] (11) -- (12);
						\end{tikzpicture} \\ [2pt]
						\hline \\[-16pt]
						{$\boxed{\smash{b_3}\vphantom{a}}$}&
						\begin{tikzpicture}[baseline={([yshift=-.5ex]current bounding box.center)}, scale=0.35,  every node/.style={scale=0.5}]
							\node[white] (1-2) at (3,0) {};
							\node[white] (2-2) at (4,1) {};
							\node[white] (3-2) at (5,0) {};
							\node[black] (1-1) at (6,1) {};
							\node (7) at (6.7,0) {};
							\node (8) at (7.4,1) {};
							\node (9) at (8.1,0) {};
							\node (10) at (8.8,1) {};
							\node (11) at (9.5,0) {};
							\node[black] (12) at (10.2,1) {};
							\draw[->] (1-2) -- (2-2);
							\draw[->] (2-2) -- (3-2);
							\draw[->] (3-2) -- (1-1);
							\draw[->] (1-1) -- (7);
							\draw[->] (7) -- (8);
							\draw[->] (8) -- (9);
							\draw[->] (9) -- (10);
							\draw[->] (10) -- (11);
							\draw[->] (11) -- (12);
						\end{tikzpicture} \\ [2pt]
						\hline \\[-16pt]
						{$\boxed{\smash{b_4}\vphantom{a}}$}&
						\begin{tikzpicture}[baseline={([yshift=-.5ex]current bounding box.center)}, scale=0.35,  every node/.style={scale=0.5}]
							\node[black] (1-2) at (3,0) {};
							\node (2-2) at (4,1) {};
							\node[white] (3-2) at (5,0) {};
							\node[black] (1-1) at (6,1) {};
							\node (7) at (6.7,0) {};
							\node (8) at (7.4,1) {};
							\node (9) at (8.1,0) {};
							\node (10) at (8.8,1) {};
							\node (11) at (9.5,0) {};
							\node[black] (12) at (10.2,1) {};
							\draw[->] (1-2) -- (2-2);
							\draw[->] (2-2) -- (3-2);
							\draw[->] (3-2) -- (1-1);
							\draw[->] (1-1) -- (7);
							\draw[->] (7) -- (8);
							\draw[->] (8) -- (9);
							\draw[->] (9) -- (10);
							\draw[->] (10) -- (11);
							\draw[->] (11) -- (12);
						\end{tikzpicture} \\ [2pt]
						\hline	
				\end{tabular}}
			\end{picture}
			\caption{Some torsion pairs in $\D$ and $\widetilde{\tors}\clH_{[t_1,t_2]}$}\label{figure4}
		\end{figure}
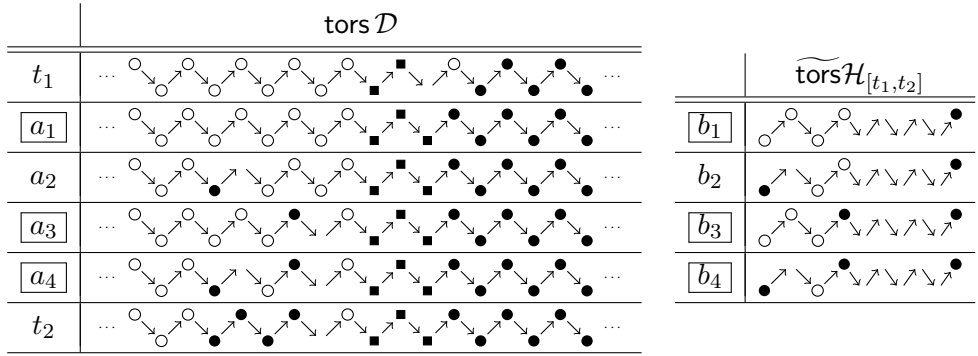
		
	\end{Ex}
	
	\section{Generalized HRS tilting on extended hearts}
	Throughout this section, let $\D$ be a triangulated category equipped with a $t$-structure $\clU=(\Ul0,\Ug0)$ and $\clH$ is the heart of $\clU$. For any positive integer $m$, there exists a $t$-structure $\clU[m]=(\Ul{-m},\Ug{-m})$ on $\D$ such that $\clU[m] \preccurlyeq \clU$. Denote $\clU^{[p,q]}=\Ug p \cap \Ul q$ for two arbitrary integers $p \leqslant q$. By \Cref{corollary3.11} (2), we have the following proposition.
	
	\begin{Prop}[{\cite[Proposition 1.9]{Zhou24}}]\label{proposition4.1}
		Let $\D$ be a triangulated category with a $t$-structure  $\clU=(\Ul0,\Ug0)$ and $\Hm=\Ul0 \cap \Ug{1-m}$. There are the following order preserving, mutually inverse bijections 
		$$\tstr[\clU[m],\clU] \overset{\phi}{\underset{\psi}{\rightleftarrows}} \stors\Hm$$ 
		where $\phi(\Vl0,\Vg0)=(\Vl0\cap\Ug{1-m},\Vg1\cap \Ul0)$ and $\psi(\T,\F)=(\Ul{-m}*\T,\F[1]*\Ug0)$.
	\end{Prop}
	
	\begin{Remark}
		We observe that $$\Hm=\clU^{[-(m-1),0]}=\clH[m-1]*\clH[m-2]*\cdots *\clH.$$ We call $\Hm$ the \emph{$m$-extended heart} of $\clU$; see \cite[Definition 1.3]{Zhou24}.
	\end{Remark}
	
	It is obvious that the $m$-extended heart $\Hm$ is closed under direct summands and extensions. Since $\Hm$ is an extension-closed subcategory of $\D$, it can be regarded as an extriangulated category with the negative first extension $\mathbb{E}^{-1}(-,-)=\D(-,-[-1])$. Moreover, $\Hm$ admits further useful properties.
	
	\begin{Lemma}\label{lemma4.3}
		Let $\Hm=\Ug{-(m-1)} \cap \Ul0$ be an $m$-extended heart on $\D$. Then the following holds.
		\begin{enumerate}
			\item $\D(\Hm,\Hm[-m-k])=0$ for any $k \geqslant 0$.
			\item $\D(\Hm[m],\Hm[-1]*\Hm[-m-1]*\Hm)=0$.
			\item $(\Hm[m]*\Hm)\cap (\Hm*\Hm[-m])=\Hm$.
			\item $(\Hm[m]*\Hm*\Hm[-m])\cap (\Hm*\Hm[-m]*\Hm[-2m])=\Hm*\Hm[-m]$.
		\end{enumerate}
	\end{Lemma}
	\begin{proof}
		\par (1) Since $\Hm \subseteq \Ul0$ and $\Hm[-m-k] \subseteq \Ug{k+1}$, we have $\D(\Hm,\Hm[-m-k])=0$.
		\par (3) By definition, we have 
		\begin{align*}
			\Hm[m]*\Hm&=(\clH[2m-1]*\clH[2m-2]*\cdots*\clH[m])*(\clH[m-1]*\clH[m-2]*\cdots*\clH)\\
			&=2m\text{-}\mathcal{H}=\Ug{-(2m-1)}\cap\Ul0.
		\end{align*}
		Similarly, we have $\Hm*\Hm[-m]=\Ug{-(m-1)}\cap\Ul{m}$. Hence $$(\Hm[m]*\Hm)\cap (\Hm*\Hm[-m])=\Ug{-(m-1)}\cap\Ul0=\Hm.$$
		\par The proofs of (2) and (4) are similar.
	\end{proof}
	
	Inspired by \cite{Jorgensen21}, we define a binary relation on the $m$-extended hearts of $\D$.
	
	\begin{Not}\label{notation4.4}
		Let $\Em_1$ and $\Em_2$ be two $m$-extended hearts of $\D$. We write $\Em_1 \leqslant \Em_2$ if and only if 
		$$\Em_1 \subseteq \Em_2[m]*\Em_2 \quad \text{and} \quad \Em_2 \subseteq \Em_1*\Em_1[-m].$$
	\end{Not}
	
	Note that $\Em_1 \leqslant \Em_2$ if and only if $\Em_2 \leqslant \Em_1[-m]$. 
	
	\begin{Ex}\label{example4.5}
		Let $(\T,\F)$ be an $s$-torsion pair in $\Hm$. Then $\F[m]*\T$ is an $m$-extended heart on $\D$ by  \cite[Theorem 1.12]{Zhou24}. We claim that $\F[m]*\T \leqslant \Hm$. Indeed, $\F[m]*\T\subseteq \Hm[m]*\Hm$, and $\Hm=\T*\F \subseteq (\F[m]*\T)*((\F[m]*\T)[-m])$.
	\end{Ex}
	
	For convenience, we define the set
	\begin{align*}
		[\Hm[m],\Hm] &:= \left\{ \text{$m$-extended hearts } \Em \text{ on } \D \;\middle|\; \Hm[m] \leqslant\Em \leqslant \Hm \right\}\\&= \left\{ \text{$m$-extended hearts } \Em \text{ on } \D \;\middle|\; \Em \leqslant \Hm \right\}.
	\end{align*}

	\begin{Lemma}
		$\leqslant$ is a partial order on $[\Hm[m],\Hm]$.
	\end{Lemma} 
	\begin{proof}
		It is clear that $\leqslant$ is reflexive.
		\par To prove that $\leqslant$ is antisymmetric. Suppose that $\Em_1,\Em_2 \in [\Hm[m],\Hm]$ such that $\Em_1 \leqslant \Em_2$ and $\Em_2 \leqslant \Em_1$. Then by definition, $$\Em_2 \subseteq \Em_1*\Em_1[-m] \quad \text{and} \quad \Em_2 \subseteq \Em_1[m]*\Em_1.$$ It follows that $\Em_2 \subseteq (\Em_1[m]*\Em_1)\cap(\Em_1*\Em_1[-m])$. By \Cref{lemma4.3} (3), $\Em_2 \subseteq \Em_1$. Similarly, $\Em_1 \subseteq \Em_2$. Hence $\Em_1=\Em_2$.
		\par To prove that $\leqslant$ is transitive. Let $\Em_1,\Em_2,\Em_3 \in [\Hm[m],\Hm]$ such that $\Em_1\leqslant\Em_2$ and	$\Em_2\leqslant\Em_3$. Then 
		\begin{align*}
			\Em_3 &\subseteq \Em_2*\Em_2[-m]\\ &\subseteq (\Em_1*\Em_1[-m])*(\Em_1[-m]*\Em_1[-2m])\\ &=\Em_1*\Em_1[-m]*\Em_1[-2m],\\
			\Em_3 &\subseteq \Hm[m]*\Hm\\ &=(\Em_1[m]*\Em_1)*(\Em_1*\Em_1[-m])\\ &=\Em_1[m]*\Em_1*\Em_1[-m].
		\end{align*} 
		Therefore, $$\Em_3 \subseteq (\Em_1*\Em_1[-m]*\Em_1[-2m])\cap(\Em_1[m]*\Em_1*\Em_1[-m]).$$
		By \Cref{lemma4.3} (4), we obtain $\Em_3 \subseteq \Em_1*\Em_1[-m]$. Similar computations show $\Em_1 \subseteq \Em_3[m]*\Em_3$, hence $\Em_1 \leqslant \Em_3$.
	\end{proof}
	
	\begin{Lemma}\label{lemma4.7}
		Let $\Em$ and  $\Hm$ be $m$-extended hearts on $\D$ such that $\Em 
		\leqslant \Hm$. Then $(\Hm\cap\Em,\Hm\cap\Em[-m])$ is an $s$-torsion pair in $\Hm$.
	\end{Lemma}
	\begin{proof}
		(TP1) follows from \Cref{lemma4.3} (1). Now we verify (TP2). Since $\Hm \subseteq \Em*\Em[-m]$, for any $X \in \Hm$, there exists a distinguished triangle 
		$$A \rightarrow X \rightarrow B[-m] \rightarrow A[1]$$
		where $A,B \in \Em$. By rotation, we have $A \in \Em[-m-1]*\Hm \subseteq \Hm[-1]*\Hm[-m-1]*\Hm$. We now show that $A \in \Hm$. Since $A \in \Em \subseteq \Hm[m]*\Hm$, we obtain a distinguished triangle 
		$$C[m] \stackrel{u}{\rightarrow} A \rightarrow D \rightarrow C[m+1]$$
		where $C,D \in \Hm$. By \Cref{lemma4.3} (2), it follows that $u=0$. Hence, $A$ is a direct summand of $D$, and so $A\in \Hm$. Therefore, $A \in \Hm \cap \Em$. Similarly, $B[-m] \in \Hm \cap \Em[-m]$. This shows that $\Hm \subseteq (\Hm\cap\Em)*(\Hm\cap\Em[-m])$. Conversely, since $\Hm$ is extension-closed, we immediately have $\Hm \supseteq (\Hm\cap\Em)*(\Hm\cap\Em[-m])$. Hence, $\Hm = (\Hm\cap\Em)*(\Hm\cap\Em[-m])$. Moreover, it is straightforward to verify that	$\D(\Hm\cap\Em,(\Hm\cap\Em[-m])[-1])=0$ by \Cref{lemma4.3} (1). (STP) holds.
	\end{proof}
	
	Now we can finish the proof of \Cref{theorem1.3}.
	
	\begin{Thm}\label{theorem4.8}
		Let $\Hm$ be an $m$-extended heart on $\D$. Then there are order preserving, mutually inverse bijections
		$$\xymatrix{\stors\Hm \ar@<0.5ex>[rr]^{\phi\quad} && [\Hm[m],\Hm]\ar@<0.5ex>[ll]^{\psi\quad}},$$
		given by
		$$\phi(\T,\F)=\F[m]*\T,\quad\quad \psi(\Em)=(\Hm\cap\Em,\Hm\cap\Em[-m]).$$
	\end{Thm}
	\begin{proof}
		\par The maps $\phi$ and $\psi$ are well defined by \Cref{example4.5} and \Cref{lemma4.7}.
		\par We want to show that $\phi\psi=\id_{[\Hm[m],\Hm]}$. As mentioned before, if $\Em \leqslant \Hm$ are two $m$-extended hearts on $\D$, then $\Hm \leqslant \Em[-m]$. By \Cref{lemma4.7}, $(\Em[-m]\cap\Hm,\Em[-m]\cap\Hm[-m])$ is an $s$-torsion pair in $\Em[-m]$, hence $(\Em[-m]\cap\Hm)*(\Em[-m]\cap\Hm[-m])=\Em[-m]$. So $(\Em\cap\Hm[m])*(\Em\cap\Hm)=\Em$.
		\par We want to show that $\psi\phi=\id_{\stors\Hm}$, that is, for any $(\T,\F) \in \stors\Hm$, $\T=\Hm\cap(\F[m]*\T)$ and $\F=\Hm\cap(\F*\T[-m])$.	We only prove the former, and the latter is analogous. It is clear that $\T \subseteq \Hm\cap(\F[m]*\T)$. Conversely, let $X \in \Hm \cap (\F[m]*\T)$. Then there exists a distinguished triangle 
		$$A\stackrel{u}{\rightarrow}X\rightarrow B \rightarrow A[1],$$
		where $A\in \F[m]\subseteq\Hm[m]$ and $B \in \T$. It follows that $u=0$ by \Cref{lemma4.3} (1), so $X$ is a direct summand of  $B$ and $X \in \T$.
		\par Let $t_1=(\T_1,\F_1)$, $t_2=(\T_2,\F_2)\in \stors\Hm$ such that $t_1 \preccurlyeq t_2$. We want to prove that $\phi$ preserves the partial orders, that is, 
		\begin{align*}
			&\F_1[m]*\T_1 \subseteq (\F_2[2m]*\T_2[m])*(\F_2[m]*\T_2), \\
			&\F_2[m]*\T_2 \subseteq (\F_1[m]*\T_1)*(\F_1*\T_1[-m]).
		\end{align*}
		Indeed, since $\F_1[m]\subseteq \Hm[m]\subseteq \F_2[2m]*(\T_2[m]*\F_2[m])$, $\T_2\subseteq \Hm \subseteq (\T_1*\F_1)*\T_1[-m]$ and $t_1 \preccurlyeq t_2$, the above inclusions follow immediately.
		\par Let $\Em_1\,\Em_2 \in [\Hm[m],\Hm]$ such that $\Em_1\leqslant\Em_2$. We want to prove that $\psi$ preserves the partial orders, that is, $$\Hm\cap\Em_1\subseteq \Hm\cap\Em_2.$$ 
		Let $X\in \Hm\cap\Em_1$. Since $\Hm\subseteq\Em_2*\Em_2[-m]$, there exists a distinguished triangle 
		$$A\rightarrow X \stackrel{v}{\rightarrow} B[-m] \rightarrow A[1],$$
		where $A,B \in \Em_2$. Moreover, since $X\in \Em_1 \subseteq \Em_2[m]*\Em_2$, it follows from \Cref{lemma4.3} (1) that $v=0$. Hence $X$ is a direct summand of $A$ and so $X \in \Em_2$. Therefore $t_1 \preccurlyeq t_2$.
	\end{proof}
	
	\begin{Ex}
		Let $\D$ be the triangulated category in \Cref{example3.13} and fix $m=2$. Let $\clU=(\Ul0,\Ug0)$ be the standard $t$-structure on $\D$  with heart $\clH=\mod kQ$. By computation, we get the Hasse quiver of the $s$-torsion pairs in $2\text{-}\clH$ in \Cref{figure5}. We depict the $s$-torsion pairs in the Auslander-Reiten quiver by marking in black the indecomposable objects in the torsion classes, in white the indecomposable objects in the torsionfree classes. The square nodes belong to $\mod kQ$.
		\begin{figure}[htbp]
			\begin{tikzpicture}[yscale=-1] 
				\node at (100pt,0pt){
					\begin{tikzpicture}[scale=0.35, every node/.style={scale=0.5}]
						\node (1) at (0,0) [b] {};
						\node (2) at (1,1) [b] {};
						\node (3) at (2,0) [b] {};
						\node (4) at (3,1) [black] {};
						\node (5) at (4,0) [black] {};
						\node (6) at (5,1) [black] {};
						\draw[->] (1) -- (2);
						\draw[->] (2) -- (3);
						\draw[->] (3) -- (4);
						\draw[->] (4) -- (5);
						\draw[->] (5) -- (6);
				\end{tikzpicture}} ;
				\node at (100pt,30pt){
					\begin{tikzpicture}[scale=0.35, every node/.style={scale=0.5}]
						\node (1) at (0,0) [w] {};
						\node (2) at (1,1) [b] {};
						\node (3) at (2,0) [b] {};
						\node (4) at (3,1) [black] {};
						\node (5) at (4,0) [black] {};
						\node (6) at (5,1) [black] {};
						\draw[->] (1) -- (2);
						\draw[->] (2) -- (3);
						\draw[->] (3) -- (4);
						\draw[->] (4) -- (5);
						\draw[->] (5) -- (6);
				\end{tikzpicture} } ;
				\node at (0pt,45pt){
					\begin{tikzpicture}[scale=0.35, every node/.style={scale=0.5}]
						\node (1) at (0,0) [b] {};
						\node (2) at (1,1) [] {};
						\node (3) at (2,0) [w] {};
						\node (4) at (3,1) [black] {};
						\node (5) at (4,0) [black] {};
						\node (6) at (5,1) [black] {};
						\draw[->] (1) -- (2);
						\draw[->] (2) -- (3);
						\draw[->] (3) -- (4);
						\draw[->] (4) -- (5);
						\draw[->] (5) -- (6);
				\end{tikzpicture} } ;
				\node at (100pt,60pt){
					\begin{tikzpicture}[scale=0.35, every node/.style={scale=0.5}]
						\node (1) at (0,0) [w] {};
						\node (2) at (1,1) [w] {};
						\node (3) at (2,0) [b] {};
						\node (4) at (3,1) [black] {};
						\node (5) at (4,0) [black] {};
						\node (6) at (5,1) [black] {};
						\draw[->] (1) -- (2);
						\draw[->] (2) -- (3);
						\draw[->] (3) -- (4);
						\draw[->] (4) -- (5);
						\draw[->] (5) -- (6);
				\end{tikzpicture} } ;
				\node at (100pt,90pt){
					\begin{tikzpicture}[scale=0.35, every node/.style={scale=0.5}]
						\node (1) at (0,0) [w] {};
						\node (2) at (1,1) [w] {};
						\node (3) at (2,0) [w] {};
						\node (4) at (3,1) [black] {};
						\node (5) at (4,0) [black] {};
						\node (6) at (5,1) [black] {};
						\draw[->] (1) -- (2);
						\draw[->] (2) -- (3);
						\draw[->] (3) -- (4);
						\draw[->] (4) -- (5);
						\draw[->] (5) -- (6);
				\end{tikzpicture} } ;
				\node at (200pt,75pt){
					\begin{tikzpicture}[scale=0.35, every node/.style={scale=0.5}]
						\node (1) at (0,0) [w] {};
						\node (2) at (1,1) [b] {};
						\node (3) at (2,0) [] {};
						\node (4) at (3,1) [white] {};
						\node (5) at (4,0) [black] {};
						\node (6) at (5,1) [black] {};
						\draw[->] (1) -- (2);
						\draw[->] (2) -- (3);
						\draw[->] (3) -- (4);
						\draw[->] (4) -- (5);
						\draw[->] (5) -- (6);
				\end{tikzpicture} } ;
				\node at (0pt,90pt){
					\begin{tikzpicture}[scale=0.35, every node/.style={scale=0.5}]
						\node (1) at (0,0) [b] {};
						\node (2) at (1,1) [] {};
						\node (3) at (2,0) [w] {};
						\node (4) at (3,1) [black] {};
						\node (5) at (4,0) [] {};
						\node (6) at (5,1) [white] {};
						\draw[->] (1) -- (2);
						\draw[->] (2) -- (3);
						\draw[->] (3) -- (4);
						\draw[->] (4) -- (5);
						\draw[->] (5) -- (6);
				\end{tikzpicture} } ;
				\node at (100pt,120pt){
					\begin{tikzpicture}[scale=0.35, every node/.style={scale=0.5}]
						\node (1) at (0,0) [w] {};
						\node (2) at (1,1) [w] {};
						\node (3) at (2,0) [w] {};
						\node (4) at (3,1) [white] {};
						\node (5) at (4,0) [black] {};
						\node (6) at (5,1) [black] {};
						\draw[->] (1) -- (2);
						\draw[->] (2) -- (3);
						\draw[->] (3) -- (4);
						\draw[->] (4) -- (5);
						\draw[->] (5) -- (6);
				\end{tikzpicture} } ;
				\node at (200pt,105pt){
					\begin{tikzpicture}[scale=0.35, every node/.style={scale=0.5}]
						\node (1) at (0,0) [w] {};
						\node (2) at (1,1) [w] {};
						\node (3) at (2,0) [b] {};
						\node (4) at (3,1) [] {};
						\node (5) at (4,0) [white] {};
						\node (6) at (5,1) [black] {};
						\draw[->] (1) -- (2);
						\draw[->] (2) -- (3);
						\draw[->] (3) -- (4);
						\draw[->] (4) -- (5);
						\draw[->] (5) -- (6);
				\end{tikzpicture} } ;
				\node at (0pt,135pt){
					\begin{tikzpicture}[scale=0.35, every node/.style={scale=0.5}]
						\node (1) at (0,0) [w] {};
						\node (2) at (1,1) [w] {};
						\node (3) at (2,0) [w] {};
						\node (4) at (3,1) [black] {};
						\node (5) at (4,0) [] {};
						\node (6) at (5,1) [white] {};
						\draw[->] (1) -- (2);
						\draw[->] (2) -- (3);
						\draw[->] (3) -- (4);
						\draw[->] (4) -- (5);
						\draw[->] (5) -- (6);
				\end{tikzpicture} } ;
				\node at (100pt,150pt){
					\begin{tikzpicture}[scale=0.35, every node/.style={scale=0.5}]
						\node (1) at (0,0) [w] {};
						\node (2) at (1,1) [w] {};
						\node (3) at (2,0) [w] {};
						\node (4) at (3,1) [white] {};
						\node (5) at (4,0) [white] {};
						\node (6) at (5,1) [black] {};
						\draw[->] (1) -- (2);
						\draw[->] (2) -- (3);
						\draw[->] (3) -- (4);
						\draw[->] (4) -- (5);
						\draw[->] (5) -- (6);
				\end{tikzpicture} } ;
				\node at (100pt,180pt){
					\begin{tikzpicture}[scale=0.35, every node/.style={scale=0.5}]
						\node (1) at (0,0) [w] {};
						\node (2) at (1,1) [w] {};
						\node (3) at (2,0) [w] {};
						\node (4) at (3,1) [white] {};
						\node (5) at (4,0) [white] {};
						\node (6) at (5,1) [white] {};
						\draw[->] (1) -- (2);
						\draw[->] (2) -- (3);
						\draw[->] (3) -- (4);
						\draw[->] (4) -- (5);
						\draw[->] (5) -- (6);
				\end{tikzpicture}};
				\draw[->, shorten >=10pt, shorten <=10pt] (100pt,0pt)--(100pt,30pt);
				\draw[->, shorten >=22pt, shorten <=22pt] (100pt,0pt)--(0pt,45pt);
				\draw[->, shorten >=10pt, shorten <=10pt] (100pt,30pt)--(100pt,60pt);
				\draw[->, shorten >=22pt, shorten <=22pt] (100pt,30pt)--(200pt,75pt);
				\draw[->, shorten >=22pt, shorten <=22pt] (0pt,45pt)--(100pt,90pt);
				\draw[->, shorten >=10pt, shorten <=10pt] (0pt,45pt)--(0pt,90pt);
				\draw[->, shorten >=10pt, shorten <=10pt] (100pt,60pt)--(100pt,90pt);
				\draw[->, shorten >=22pt, shorten <=22pt] (100pt,60pt)--(200pt,105pt);
				\draw[->, shorten >=10pt, shorten <=10pt] (100pt,90pt)--(100pt,120pt);
				\draw[->, shorten >=22pt, shorten <=22pt] (100pt,90pt)--(0pt,135pt);
				\draw[->, shorten >=22pt, shorten <=22pt] (200pt,75pt)--(100pt,120pt);
				\draw[->, shorten >=10pt, shorten <=10pt] (0pt,90pt)--(0pt,135pt);
				\draw[->, shorten >=10pt, shorten <=10pt] (100pt,120pt)--(100pt,150pt);
				\draw[->, shorten >=22pt, shorten <=22pt] (200pt,105pt)--(100pt,150pt);
				\draw[->, shorten >=22pt, shorten <=22pt] (0pt,135pt)--(100pt,180pt);
				\draw[->, shorten >=10pt, shorten <=10pt] (100pt,150pt)--(100pt,180pt);
			\end{tikzpicture}
			\caption{The Hasse quiver of $s$-torsion pairs in 2-$\clH$}\label{figure5}
		\end{figure}
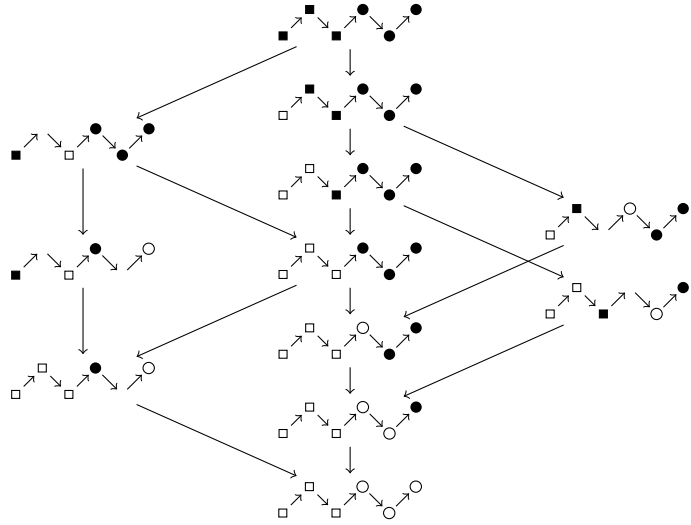
		
		Similarly, \Cref{figure6} shows the Hasse quiver of the 2-extended hearts in $[2\text{-}\clH[2],2\text{-}\clH]$, where indecomposable objects in the extended hearts are marked in black, and those in $\mod kQ$ are marked as squares. Here, $2\text{-}\clH$ is the top extended heart and $2\text{-}\clH[2]$ is the bottom one.
		\begin{figure}[htbp]
			\begin{tikzpicture}[yscale=-1] 
				\node at (150pt,0pt){
					\begin{tikzpicture}[scale=0.35, every node/.style={scale=0.5}]
						\node (1) at (0,0) [b] {};
						\node (2) at (1,1) [b] {};
						\node (3) at (2,0) [b] {};
						\node (4) at (3,1) [black] {};
						\node (5) at (4,0) [black] {};
						\node (6) at (5,1) [black] {};
						\node (7) at (6,0) [white] {};
						\node (8) at (7,1) [white] {};
						\node (9) at (8,0) [white] {};
						\node (10) at (9,1) [white] {};
						\node (11) at (10,0) [white] {};
						\node (12) at (11,1) [white] {};
						\draw[->] (1) -- (2);
						\draw[->] (2) -- (3);
						\draw[->] (3) -- (4);
						\draw[->] (4) -- (5);
						\draw[->] (5) -- (6);
						\draw[->] (6) -- (7);
						\draw[->] (7) -- (8);
						\draw[->] (8) -- (9);
						\draw[->] (9) -- (10);
						\draw[->] (10) -- (11);
						\draw[->] (11) -- (12);
				\end{tikzpicture}} ;
				\node at (150pt,30pt){
					\begin{tikzpicture}[scale=0.35, every node/.style={scale=0.5}]
						\node (1) at (0,0) [w] {};
						\node (2) at (1,1) [b] {};
						\node (3) at (2,0) [b] {};
						\node (4) at (3,1) [black] {};
						\node (5) at (4,0) [black] {};
						\node (6) at (5,1) [black] {};
						\node (7) at (6,0) [black] {};
						\node (8) at (7,1) [white] {};
						\node (9) at (8,0) [white] {};
						\node (10) at (9,1) [white] {};
						\node (11) at (10,0) [white] {};
						\node (12) at (11,1) [white] {};
						\draw[->] (1) -- (2);
						\draw[->] (2) -- (3);
						\draw[->] (3) -- (4);
						\draw[->] (4) -- (5);
						\draw[->] (5) -- (6);
						\draw[->] (6) -- (7);
						\draw[->] (7) -- (8);
						\draw[->] (8) -- (9);
						\draw[->] (9) -- (10);
						\draw[->] (10) -- (11);
						\draw[->] (11) -- (12);
				\end{tikzpicture}} ;
				\node at (0pt,45pt){
					\begin{tikzpicture}[scale=0.35, every node/.style={scale=0.5}]
						\node (1) at (0,0) [b] {};
						\node (2) at (1,1) [w] {};
						\node (3) at (2,0) [w] {};
						\node (4) at (3,1) [black] {};
						\node (5) at (4,0) [black] {};
						\node (6) at (5,1) [black] {};
						\node (7) at (6,0) [white] {};
						\node (8) at (7,1) [white] {};
						\node (9) at (8,0) [black] {};
						\node (10) at (9,1) [white] {};
						\node (11) at (10,0) [white] {};
						\node (12) at (11,1) [white] {};
						\draw[->] (1) -- (2);
						\draw[->] (2) -- (3);
						\draw[->] (3) -- (4);
						\draw[->] (4) -- (5);
						\draw[->] (5) -- (6);
						\draw[->] (6) -- (7);
						\draw[->] (7) -- (8);
						\draw[->] (8) -- (9);
						\draw[->] (9) -- (10);
						\draw[->] (10) -- (11);
						\draw[->] (11) -- (12);
				\end{tikzpicture}} ;
				\node at (150pt,60pt){
					\begin{tikzpicture}[scale=0.35, every node/.style={scale=0.5}]
						\node (1) at (0,0) [w] {};
						\node (2) at (1,1) [w] {};
						\node (3) at (2,0) [b] {};
						\node (4) at (3,1) [black] {};
						\node (5) at (4,0) [black] {};
						\node (6) at (5,1) [black] {};
						\node (7) at (6,0) [black] {};
						\node (8) at (7,1) [black] {};
						\node (9) at (8,0) [white] {};
						\node (10) at (9,1) [white] {};
						\node (11) at (10,0) [white] {};
						\node (12) at (11,1) [white] {};
						\draw[->] (1) -- (2);
						\draw[->] (2) -- (3);
						\draw[->] (3) -- (4);
						\draw[->] (4) -- (5);
						\draw[->] (5) -- (6);
						\draw[->] (6) -- (7);
						\draw[->] (7) -- (8);
						\draw[->] (8) -- (9);
						\draw[->] (9) -- (10);
						\draw[->] (10) -- (11);
						\draw[->] (11) -- (12);
				\end{tikzpicture}} ;
				\node at (150pt,90pt){
					\begin{tikzpicture}[scale=0.35, every node/.style={scale=0.5}]
						\node (1) at (0,0) [w] {};
						\node (2) at (1,1) [w] {};
						\node (3) at (2,0) [w] {};
						\node (4) at (3,1) [black] {};
						\node (5) at (4,0) [black] {};
						\node (6) at (5,1) [black] {};
						\node (7) at (6,0) [black] {};
						\node (8) at (7,1) [black] {};
						\node (9) at (8,0) [black] {};
						\node (10) at (9,1) [white] {};
						\node (11) at (10,0) [white] {};
						\node (12) at (11,1) [white] {};
						\draw[->] (1) -- (2);
						\draw[->] (2) -- (3);
						\draw[->] (3) -- (4);
						\draw[->] (4) -- (5);
						\draw[->] (5) -- (6);
						\draw[->] (6) -- (7);
						\draw[->] (7) -- (8);
						\draw[->] (8) -- (9);
						\draw[->] (9) -- (10);
						\draw[->] (10) -- (11);
						\draw[->] (11) -- (12);
				\end{tikzpicture}	} ;
				\node at (300pt,75pt){
					\begin{tikzpicture}[scale=0.35, every node/.style={scale=0.5}]
						\node (1) at (0,0) [w] {};
						\node (2) at (1,1) [b] {};
						\node (3) at (2,0) [w] {};
						\node (4) at (3,1) [white] {};
						\node (5) at (4,0) [black] {};
						\node (6) at (5,1) [black] {};
						\node (7) at (6,0) [black] {};
						\node (8) at (7,1) [white] {};
						\node (9) at (8,0) [white] {};
						\node (10) at (9,1) [black] {};
						\node (11) at (10,0) [white] {};
						\node (12) at (11,1) [white] {};
						\draw[->] (1) -- (2);
						\draw[->] (2) -- (3);
						\draw[->] (3) -- (4);
						\draw[->] (4) -- (5);
						\draw[->] (5) -- (6);
						\draw[->] (6) -- (7);
						\draw[->] (7) -- (8);
						\draw[->] (8) -- (9);
						\draw[->] (9) -- (10);
						\draw[->] (10) -- (11);
						\draw[->] (11) -- (12);
				\end{tikzpicture}} ;
				\node at (0pt,90pt){
					\begin{tikzpicture}[scale=0.35, every node/.style={scale=0.5}]
						\node (1) at (0,0) [b] {};
						\node (2) at (1,1) [w] {};
						\node (3) at (2,0) [w] {};
						\node (4) at (3,1) [black] {};
						\node (5) at (4,0) [white] {};
						\node (6) at (5,1) [white] {};
						\node (7) at (6,0) [white] {};
						\node (8) at (7,1) [white] {};
						\node (9) at (8,0) [black] {};
						\node (10) at (9,1) [white] {};
						\node (11) at (10,0) [white] {};
						\node (12) at (11,1) [black] {};
						\draw[->] (1) -- (2);
						\draw[->] (2) -- (3);
						\draw[->] (3) -- (4);
						\draw[->] (4) -- (5);
						\draw[->] (5) -- (6);
						\draw[->] (6) -- (7);
						\draw[->] (7) -- (8);
						\draw[->] (8) -- (9);
						\draw[->] (9) -- (10);
						\draw[->] (10) -- (11);
						\draw[->] (11) -- (12);
				\end{tikzpicture}} ;
				\node at (150pt,120pt){
					\begin{tikzpicture}[scale=0.35, every node/.style={scale=0.5}]
						\node (1) at (0,0) [w] {};
						\node (2) at (1,1) [w] {};
						\node (3) at (2,0) [w] {};
						\node (4) at (3,1) [white] {};
						\node (5) at (4,0) [black] {};
						\node (6) at (5,1) [black] {};
						\node (7) at (6,0) [black] {};
						\node (8) at (7,1) [black] {};
						\node (9) at (8,0) [black] {};
						\node (10) at (9,1) [black] {};
						\node (11) at (10,0) [white] {};
						\node (12) at (11,1) [white] {};
						\draw[->] (1) -- (2);
						\draw[->] (2) -- (3);
						\draw[->] (3) -- (4);
						\draw[->] (4) -- (5);
						\draw[->] (5) -- (6);
						\draw[->] (6) -- (7);
						\draw[->] (7) -- (8);
						\draw[->] (8) -- (9);
						\draw[->] (9) -- (10);
						\draw[->] (10) -- (11);
						\draw[->] (11) -- (12);
				\end{tikzpicture}} ;
				\node at (300pt,105pt){
					\begin{tikzpicture}[scale=0.35, every node/.style={scale=0.5}]
						\node (1) at (0,0) [w] {};
						\node (2) at (1,1) [w] {};
						\node (3) at (2,0) [b] {};
						\node (4) at (3,1) [white] {};
						\node (5) at (4,0) [white] {};
						\node (6) at (5,1) [black] {};
						\node (7) at (6,0) [black] {};
						\node (8) at (7,1) [black] {};
						\node (9) at (8,0) [white] {};
						\node (10) at (9,1) [white] {};
						\node (11) at (10,0) [black] {};
						\node (12) at (11,1) [white] {};
						\draw[->] (1) -- (2);
						\draw[->] (2) -- (3);
						\draw[->] (3) -- (4);
						\draw[->] (4) -- (5);
						\draw[->] (5) -- (6);
						\draw[->] (6) -- (7);
						\draw[->] (7) -- (8);
						\draw[->] (8) -- (9);
						\draw[->] (9) -- (10);
						\draw[->] (10) -- (11);
						\draw[->] (11) -- (12);
				\end{tikzpicture}} ;
				\node at (0pt,135pt){
					\begin{tikzpicture}[scale=0.35, every node/.style={scale=0.5}]
						\node (1) at (0,0) [w] {};
						\node (2) at (1,1) [w] {};
						\node (3) at (2,0) [w] {};
						\node (4) at (3,1) [black] {};
						\node (5) at (4,0) [white] {};
						\node (6) at (5,1) [white] {};
						\node (7) at (6,0) [black] {};
						\node (8) at (7,1) [black] {};
						\node (9) at (8,0) [black] {};
						\node (10) at (9,1) [white] {};
						\node (11) at (10,0) [white] {};
						\node (12) at (11,1) [black] {};
						\draw[->] (1) -- (2);
						\draw[->] (2) -- (3);
						\draw[->] (3) -- (4);
						\draw[->] (4) -- (5);
						\draw[->] (5) -- (6);
						\draw[->] (6) -- (7);
						\draw[->] (7) -- (8);
						\draw[->] (8) -- (9);
						\draw[->] (9) -- (10);
						\draw[->] (10) -- (11);
						\draw[->] (11) -- (12);
				\end{tikzpicture}} ;
				\node at (150pt,150pt){
					\begin{tikzpicture}[scale=0.35, every node/.style={scale=0.5}]
						\node (1) at (0,0) [w] {};
						\node (2) at (1,1) [w] {};
						\node (3) at (2,0) [w] {};
						\node (4) at (3,1) [white] {};
						\node (5) at (4,0) [white] {};
						\node (6) at (5,1) [black] {};
						\node (7) at (6,0) [black] {};
						\node (8) at (7,1) [black] {};
						\node (9) at (8,0) [black] {};
						\node (10) at (9,1) [black] {};
						\node (11) at (10,0) [black] {};
						\node (12) at (11,1) [white] {};
						\draw[->] (1) -- (2);
						\draw[->] (2) -- (3);
						\draw[->] (3) -- (4);
						\draw[->] (4) -- (5);
						\draw[->] (5) -- (6);
						\draw[->] (6) -- (7);
						\draw[->] (7) -- (8);
						\draw[->] (8) -- (9);
						\draw[->] (9) -- (10);
						\draw[->] (10) -- (11);
						\draw[->] (11) -- (12);
				\end{tikzpicture}} ;
				\node at (150pt,180pt){
					\begin{tikzpicture}[scale=0.35, every node/.style={scale=0.5}]
						\node (1) at (0,0) [w] {};
						\node (2) at (1,1) [w] {};
						\node (3) at (2,0) [w] {};
						\node (4) at (3,1) [white] {};
						\node (5) at (4,0) [white] {};
						\node (6) at (5,1) [white] {};
						\node (7) at (6,0) [black] {};
						\node (8) at (7,1) [black] {};
						\node (9) at (8,0) [black] {};
						\node (10) at (9,1) [black] {};
						\node (11) at (10,0) [black] {};
						\node (12) at (11,1) [black] {};
						\draw[->] (1) -- (2);
						\draw[->] (2) -- (3);
						\draw[->] (3) -- (4);
						\draw[->] (4) -- (5);
						\draw[->] (5) -- (6);
						\draw[->] (6) -- (7);
						\draw[->] (7) -- (8);
						\draw[->] (8) -- (9);
						\draw[->] (9) -- (10);
						\draw[->] (10) -- (11);
						\draw[->] (11) -- (12);
				\end{tikzpicture}};
				\draw[->, shorten >=10pt, shorten <=10pt] (150pt,0pt)--(150pt,30pt);
				\draw[->, shorten >=31pt, shorten <=31pt] (150pt,0pt)--(0pt,45pt);
				\draw[->, shorten >=10pt, shorten <=10pt] (150pt,30pt)--(150pt,60pt);
				\draw[->, shorten >=31pt, shorten <=31pt] (150pt,30pt)--(300pt,75pt);
				\draw[->, shorten >=31pt, shorten <=31pt] (0pt,45pt)--(150pt,90pt);
				\draw[->, shorten >=10pt, shorten <=10pt] (0pt,45pt)--(0pt,90pt);
				\draw[->, shorten >=10pt, shorten <=10pt] (150pt,60pt)--(150pt,90pt);
				\draw[->, shorten >=31pt, shorten <=31pt] (150pt,60pt)--(300pt,105pt);
				\draw[->, shorten >=10pt, shorten <=10pt] (150pt,90pt)--(150pt,120pt);
				\draw[->, shorten >=31pt, shorten <=31pt] (150pt,90pt)--(0pt,135pt);
				\draw[->, shorten >=31pt, shorten <=31pt] (300pt,75pt)--(150pt,120pt);
				\draw[->, shorten >=10pt, shorten <=10pt] (0pt,90pt)--(0pt,135pt);
				\draw[->, shorten >=10pt, shorten <=10pt] (150pt,120pt)--(150pt,150pt);
				\draw[->, shorten >=31pt, shorten <=31pt] (300pt,105pt)--(150pt,150pt);
				\draw[->, shorten >=31pt, shorten <=31pt] (0pt,135pt)--(150pt,180pt);
				\draw[->, shorten >=10pt, shorten <=10pt] (150pt,150pt)--(150pt,180pt);
			\end{tikzpicture}
			\caption{The Hasse quiver of 2-extended hearts in $[2\text{-}\clH[2],2\text{-}\clH]$}\label{figure6}
		\end{figure}
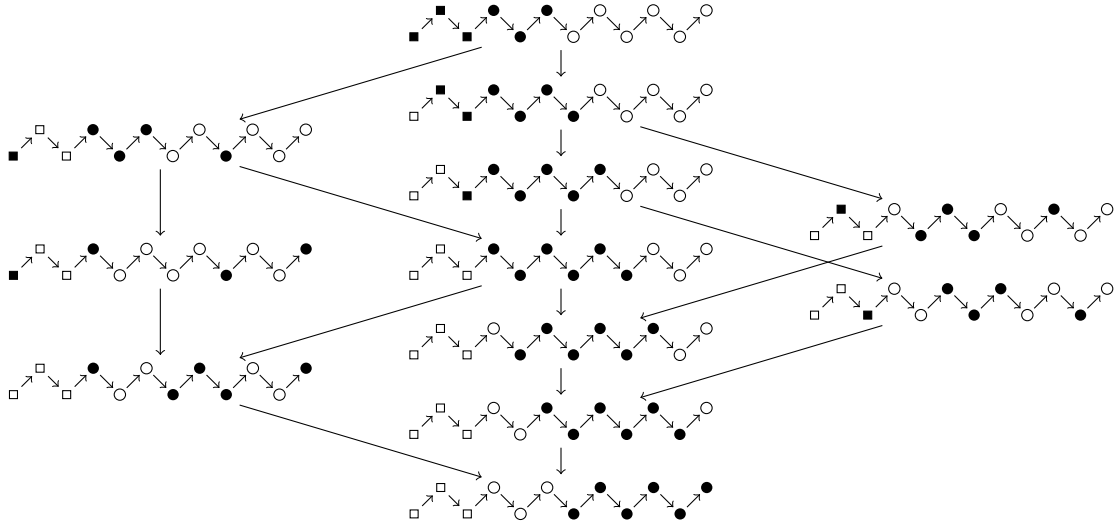
		
		The $t$-structures corresponding to the above 2-extended hearts are shown in \Cref{figure7}. Black dots represent indecomposable objects in the aisles, and square dots represent indecomposable objects in $\mod kQ $. The top $t$-structure is the standard one. These $ t$-structures are precisely those in $\tstr[\clU[2],\clU]$.
		
		\begin{figure}[htbp]
			\begin{tikzpicture}[yscale=-1] 
				\node at (125pt,0pt){
					\begin{tikzpicture}[scale=0.35, every node/.style={scale=0.5}]
						\node (1) at (0,0) [b] {};
						\node (2) at (1,1) [b] {};
						\node (3) at (2,0) [b] {};
						\node (4) at (3,1) [black] {};
						\node (5) at (4,0) [black] {};
						\node (6) at (5,1) [black] {};
						\node (7) at (6,0) [black] {};
						\node (8) at (7,1) [black] {};
						\node (9) at (8,0) [black] {};
						\node (78) at (9,0.5){$\cdots$};
						\draw[->] (1) -- (2);
						\draw[->] (2) -- (3);
						\draw[->] (3) -- (4);
						\draw[->] (4) -- (5);
						\draw[->] (5) -- (6);
						\draw[->] (6) -- (7);
						\draw[->] (7) -- (8);
						\draw[->] (8) -- (9);
				\end{tikzpicture}};
				\node at (125pt,30pt){
					\begin{tikzpicture}[scale=0.35, every node/.style={scale=0.5}]
						\node (1) at (0,0) [w] {};
						\node (2) at (1,1) [b] {};
						\node (3) at (2,0) [b] {};
						\node (4) at (3,1) [black] {};
						\node (5) at (4,0) [black] {};
						\node (6) at (5,1) [black] {};
						\node (7) at (6,0) [black] {};
						\node (8) at (7,1) [black] {};
						\node (9) at (8,0) [black] {};
						\node (78) at (9,0.5){$\cdots$};
						\draw[->] (1) -- (2);
						\draw[->] (2) -- (3);
						\draw[->] (3) -- (4);
						\draw[->] (4) -- (5);
						\draw[->] (5) -- (6);
						\draw[->] (6) -- (7);
						\draw[->] (7) -- (8);
						\draw[->] (8) -- (9);
				\end{tikzpicture}};
				\node at (0pt,45pt){
					\begin{tikzpicture}[scale=0.35, every node/.style={scale=0.5}]
						\node (1) at (0,0) [b] {};
						\node (2) at (1,1) [w] {};
						\node (3) at (2,0) [w] {};
						\node (4) at (3,1) [black] {};
						\node (5) at (4,0) [black] {};
						\node (6) at (5,1) [black] {};
						\node (7) at (6,0) [black] {};
						\node (8) at (7,1) [black] {};
						\node (9) at (8,0) [black] {};
						\node (78) at (9,0.5){$\cdots$};
						\draw[->] (1) -- (2);
						\draw[->] (2) -- (3);
						\draw[->] (3) -- (4);
						\draw[->] (4) -- (5);
						\draw[->] (5) -- (6);
						\draw[->] (6) -- (7);
						\draw[->] (7) -- (8);
						\draw[->] (8) -- (9);
				\end{tikzpicture}};
				\node at (125pt,60pt){
					\begin{tikzpicture}[scale=0.35, every node/.style={scale=0.5}]
						\node (1) at (0,0) [w] {};
						\node (2) at (1,1) [w] {};
						\node (3) at (2,0) [b] {};
						\node (4) at (3,1) [black] {};
						\node (5) at (4,0) [black] {};
						\node (6) at (5,1) [black] {};
						\node (7) at (6,0) [black] {};
						\node (8) at (7,1) [black] {};
						\node (9) at (8,0) [black] {};
						\node (78) at (9,0.5){$\cdots$};
						\draw[->] (1) -- (2);
						\draw[->] (2) -- (3);
						\draw[->] (3) -- (4);
						\draw[->] (4) -- (5);
						\draw[->] (5) -- (6);
						\draw[->] (6) -- (7);
						\draw[->] (7) -- (8);
						\draw[->] (8) -- (9);
				\end{tikzpicture}};
				\node at (125pt,90pt){
					\begin{tikzpicture}[scale=0.35, every node/.style={scale=0.5}]
						\node (1) at (0,0) [w] {};
						\node (2) at (1,1) [w] {};
						\node (3) at (2,0) [w] {};
						\node (4) at (3,1) [black] {};
						\node (5) at (4,0) [black] {};
						\node (6) at (5,1) [black] {};
						\node (7) at (6,0) [black] {};
						\node (8) at (7,1) [black] {};
						\node (9) at (8,0) [black] {};
						\node (78) at (9,0.5){$\cdots$};
						\draw[->] (1) -- (2);
						\draw[->] (2) -- (3);
						\draw[->] (3) -- (4);
						\draw[->] (4) -- (5);
						\draw[->] (5) -- (6);
						\draw[->] (6) -- (7);
						\draw[->] (7) -- (8);
						\draw[->] (8) -- (9);
				\end{tikzpicture}};
				\node at (250pt,75pt){
					\begin{tikzpicture}[scale=0.35, every node/.style={scale=0.5}]
						\node (1) at (0,0) [w] {};
						\node (2) at (1,1) [b] {};
						\node (3) at (2,0) [w] {};
						\node (4) at (3,1) [white] {};
						\node (5) at (4,0) [black] {};
						\node (6) at (5,1) [black] {};
						\node (7) at (6,0) [black] {};
						\node (8) at (7,1) [black] {};
						\node (9) at (8,0) [black] {};
						\node (78) at (9,0.5){$\cdots$};
						\draw[->] (1) -- (2);
						\draw[->] (2) -- (3);
						\draw[->] (3) -- (4);
						\draw[->] (4) -- (5);
						\draw[->] (5) -- (6);
						\draw[->] (6) -- (7);
						\draw[->] (7) -- (8);
						\draw[->] (8) -- (9);
				\end{tikzpicture}};
				\node at (0pt,90pt){
					\begin{tikzpicture}[scale=0.35, every node/.style={scale=0.5}]
						\node (1) at (0,0) [b] {};
						\node (2) at (1,1) [w] {};
						\node (3) at (2,0) [w] {};
						\node (4) at (3,1) [black] {};
						\node (5) at (4,0) [white] {};
						\node (6) at (5,1) [white] {};
						\node (7) at (6,0) [black] {};
						\node (8) at (7,1) [black] {};
						\node (9) at (8,0) [black] {};
						\node (78) at (9,0.5){$\cdots$};
						\draw[->] (1) -- (2);
						\draw[->] (2) -- (3);
						\draw[->] (3) -- (4);
						\draw[->] (4) -- (5);
						\draw[->] (5) -- (6);
						\draw[->] (6) -- (7);
						\draw[->] (7) -- (8);
						\draw[->] (8) -- (9);
				\end{tikzpicture}};
				\node at (125pt,120pt){
					\begin{tikzpicture}[scale=0.35, every node/.style={scale=0.5}]
						\node (1) at (0,0) [w] {};
						\node (2) at (1,1) [w] {};
						\node (3) at (2,0) [w] {};
						\node (4) at (3,1) [white] {};
						\node (5) at (4,0) [black] {};
						\node (6) at (5,1) [black] {};
						\node (7) at (6,0) [black] {};
						\node (8) at (7,1) [black] {};
						\node (9) at (8,0) [black] {};
						\node (78) at (9,0.5){$\cdots$};
						\draw[->] (1) -- (2);
						\draw[->] (2) -- (3);
						\draw[->] (3) -- (4);
						\draw[->] (4) -- (5);
						\draw[->] (5) -- (6);
						\draw[->] (6) -- (7);
						\draw[->] (7) -- (8);
						\draw[->] (8) -- (9);
				\end{tikzpicture}};
				\node at (250pt,105pt){
					\begin{tikzpicture}[scale=0.35, every node/.style={scale=0.5}]
						\node (1) at (0,0) [w] {};
						\node (2) at (1,1) [w] {};
						\node (3) at (2,0) [b] {};
						\node (4) at (3,1) [white] {};
						\node (5) at (4,0) [white] {};
						\node (6) at (5,1) [black] {};
						\node (7) at (6,0) [black] {};
						\node (8) at (7,1) [black] {};
						\node (9) at (8,0) [black] {};
						\node (78) at (9,0.5){$\cdots$};
						\draw[->] (1) -- (2);
						\draw[->] (2) -- (3);
						\draw[->] (3) -- (4);
						\draw[->] (4) -- (5);
						\draw[->] (5) -- (6);
						\draw[->] (6) -- (7);
						\draw[->] (7) -- (8);
						\draw[->] (8) -- (9);
				\end{tikzpicture}};
				\node at (0pt,135pt){
					\begin{tikzpicture}[scale=0.35, every node/.style={scale=0.5}]
						\node (1) at (0,0) [w] {};
						\node (2) at (1,1) [w] {};
						\node (3) at (2,0) [w] {};
						\node (4) at (3,1) [black] {};
						\node (5) at (4,0) [white] {};
						\node (6) at (5,1) [white] {};
						\node (7) at (6,0) [black] {};
						\node (8) at (7,1) [black] {};
						\node (9) at (8,0) [black] {};
						\node (78) at (9,0.5){$\cdots$};
						\draw[->] (1) -- (2);
						\draw[->] (2) -- (3);
						\draw[->] (3) -- (4);
						\draw[->] (4) -- (5);
						\draw[->] (5) -- (6);
						\draw[->] (6) -- (7);
						\draw[->] (7) -- (8);
						\draw[->] (8) -- (9);
				\end{tikzpicture}};
				\node at (125pt,150pt){
					\begin{tikzpicture}[scale=0.35, every node/.style={scale=0.5}]
						\node (1) at (0,0) [w] {};
						\node (2) at (1,1) [w] {};
						\node (3) at (2,0) [w] {};
						\node (4) at (3,1) [white] {};
						\node (5) at (4,0) [white] {};
						\node (6) at (5,1) [black] {};
						\node (7) at (6,0) [black] {};
						\node (8) at (7,1) [black] {};
						\node (9) at (8,0) [black] {};
						\node (78) at (9,0.5){$\cdots$};
						\draw[->] (1) -- (2);
						\draw[->] (2) -- (3);
						\draw[->] (3) -- (4);
						\draw[->] (4) -- (5);
						\draw[->] (5) -- (6);
						\draw[->] (6) -- (7);
						\draw[->] (7) -- (8);
						\draw[->] (8) -- (9);
				\end{tikzpicture}};
				\node at (125pt,180pt){
					\begin{tikzpicture}[scale=0.35, every node/.style={scale=0.5}]
						\node (1) at (0,0) [w] {};
						\node (2) at (1,1) [w] {};
						\node (3) at (2,0) [w] {};
						\node (4) at (3,1) [white] {};
						\node (5) at (4,0) [white] {};
						\node (6) at (5,1) [white] {};
						\node (7) at (6,0) [black] {};
						\node (8) at (7,1) [black] {};
						\node (9) at (8,0) [black] {};
						\node (78) at (9,0.5){$\cdots$};
						\draw[->] (1) -- (2);
						\draw[->] (2) -- (3);
						\draw[->] (3) -- (4);
						\draw[->] (4) -- (5);
						\draw[->] (5) -- (6);
						\draw[->] (6) -- (7);
						\draw[->] (7) -- (8);
						\draw[->] (8) -- (9);
				\end{tikzpicture}};
				\draw[->, shorten >=10pt, shorten <=10pt] (125pt,0pt)--(125pt,30pt);
				\draw[->, shorten >=26pt, shorten <=26pt] (125pt,0pt)--(0pt,45pt);
				\draw[->, shorten >=10pt, shorten <=10pt] (125pt,30pt)--(125pt,60pt);
				\draw[->, shorten >=26pt, shorten <=26pt] (125pt,30pt)--(250pt,75pt);
				\draw[->, shorten >=26pt, shorten <=26pt] (0pt,45pt)--(125pt,90pt);
				\draw[->, shorten >=10pt, shorten <=10pt] (0pt,45pt)--(0pt,90pt);
				\draw[->, shorten >=10pt, shorten <=10pt] (125pt,60pt)--(125pt,90pt);
				\draw[->, shorten >=26pt, shorten <=26pt] (125pt,60pt)--(250pt,105pt);
				\draw[->, shorten >=10pt, shorten <=10pt] (125pt,90pt)--(125pt,120pt);
				\draw[->, shorten >=26pt, shorten <=26pt] (125pt,90pt)--(0pt,135pt);
				\draw[->, shorten >=26pt, shorten <=26pt] (250pt,75pt)--(125pt,120pt);
				\draw[->, shorten >=10pt, shorten <=10pt] (0pt,90pt)--(0pt,135pt);
				\draw[->, shorten >=10pt, shorten <=10pt] (125pt,120pt)--(125pt,150pt);
				\draw[->, shorten >=26pt, shorten <=26pt] (250pt,105pt)--(125pt,150pt);
				\draw[->, shorten >=26pt, shorten <=26pt] (0pt,135pt)--(125pt,180pt);
				\draw[->, shorten >=10pt, shorten <=10pt] (125pt,150pt)--(125pt,180pt);
			\end{tikzpicture}
			\caption{The Hasse quiver of $t$-structures in $\tstr[\clU[2],\clU]$}\label{figure7}
		\end{figure}
	\end{Ex}
	
	\section{The extensions of  \texorpdfstring{$t$}{t}-structures}
	Let $\D$ be a triangulated category and $\clU=(\Ul0,\Ug0)$ be a $t$-structure with heart $\clH=\Ul0\cap\Ug0$. Let $\clS$ be a triangulated full subcategory of $\D$. By \cite{BBD82}, the following statements are equivalent.
	\begin{enumerate}
		\item $\clU_\clS=(\clS\cap\Ul0,\clS\cap\Ug0)$ is a $t$-structure on $\clS$.
		\item $\tau_{\leqslant 0}(\clS) \subseteq \clS$.
		\item $\tau_{\geqslant 0}(\clS) \subseteq \clS$.
	\end{enumerate}
	
	We wonder when a $t$-structure on $\clS$ can be  extended to a $t$-structure on $\D$.
	
	Throughout this section, let $(\Ul0,\Ug0)$ be a $t$-structure on $\D$ and $\clS$ a triangulated subcategory of $\D$ such that the heart $\clH$ is contained in $\clS$ and $\clU_{\clS}$ is a $t$-structure on $\clS$. Clearly, the heart of $\clU_{\clS}$ coincides with the heart of $\clU$, and their corresponding $m$-extended hearts also coincide, denoted by $\Hm$.
	
	Via the HRS tilting established in \Cref{proposition4.1}, we now demonstrate a one-to-one correspondence between the $t$-structures of $\D$ and $\clS$, which are linked by the same extended heart $\Hm$. This correspondence is depicted in \Cref{figure8}, where the morphisms are given by \Cref{proposition4.1}. In the following, we will calculate  $\psi\circ\phi'$ in \Cref{lemma5.1} and $\psi'\circ\phi$ in \Cref{lemma5.2}.
	
	\begin{figure}[htp]
		$$\xymatrix{\tstr[\clU[m],\clU] \ar@<0.5ex>[rr]^{\lambda} \ar@<-0.5ex>[rd]_{\phi} && \tstr[\clU_\clS [m],\clU_\clS]\ar@<0.5ex>[ll]^{\mu}\ar@<0.5ex>[ld]^{\phi'}\\&\stors \Hm\ar@<0.5ex>[ur]^{\psi'}\ar@<-0.5ex>[ul]_{\psi}&},$$
		\caption{The bijections}\label{figure8}
	\end{figure}	
	
	\begin{Lemma}\label{lemma5.1}
		Let $(\Yl0,\Yg0) \in \tstr[\clU_\clS[m],\clU_\clS]$ be a $t$-structure on $\clS$, then 
		$$\psi\circ\phi'(\Yl0,\Yg0)=(\Ul{-m}*\Yl0,\Yg0*\Ug0).$$
	\end{Lemma}
	\begin{proof}
		Since $(\Yl0,\Yg0) \in \tstr[\clU_\clS[m],\clU_\clS]$, we obtain an $s$-torsion pair $\phi'(\Yl0,\Yg0)\in\stors\Hm$ by \Cref{proposition4.1}. By \Cref{proposition4.1} again, we have a $t$-structure
		\begin{align*}
			\psi\circ\phi'(\Yl0,\Yg0)&=\psi(\Hm\cap\Yl0,\Hm[1]\cap\Yg0)\\&=(\Ul{-m}*(\Hm\cap\Yl0),(\Hm[1]\cap\Yg0)*\Ug0).
		\end{align*}
		
		Clearly, $\Ul{-m}*(\Hm\cap\Yl0) \subseteq \Ul{-m}*\Yl0$. Next, we show that $\Ul{-m}*\Yl0 \subseteq \Ul{-m}*(\Hm\cap\Yl0)$. Suppose $ X \in \Ul{-m}*\Yl0$, then there exists a distinguished triangle
		$$ A \stackrel{p}{\rightarrow}  X \rightarrow  B \rightarrow  A[1]$$
		with $ A \in \Ul{-m}$ and $ B \in \Yl0$. Since $X \in \D$, there exists a distinguished triangle
		$$ C \stackrel{r}{\rightarrow}  X \stackrel{s}{\rightarrow} D \rightarrow  D[1]$$
		with $C \in \Ul{-m}$ and $ D \in \Ug{-m+1}$. We claim that $D\in \Hm\cap\Yl0$, thus $X \in \Ul{-m}*(\Hm\cap\Yl0)$. In fact,  since $s\circ p =0$, there exists a morphism $q \in \D( A, C)$ such that $p=r\circ q$. Applying the octahedral axiom, we obtain the following commutative diagram
		$$\xymatrix{ A \ar[r]^q\ar@{=}[d] &  C \ar[r]\ar[d]^r &  E \ar[r]\ar@{-->}[d] &  A[1]\ar@{=}[d]\\  A \ar[r]^p & X \ar[r]\ar[d]^s &  B \ar[r]\ar@{-->}[d] &  A[1]\\& D\ar@{=}[r]\ar[d] & D\ar@{-->}[d] &\\& C[1]\ar[r] &  E[1] }$$
		where the first row is a distinguished triangle. Since  $A, C \in \Ul{-m}$, it follows that $ E \in \Ul{-m}$. In addition, as $D \in \Ug{-m+1}$ and $B \in \clS$, we have $ E \cong \tau_{\leqslant -m}B \in  \clS\cap\Ul{-m}\subseteq\Yl0$ and $ D \cong \tau_{\geqslant-m+1}B \in \clS\cap\Ug{-m+1}\subseteq\Hm$. Note that $ B \in  \Yl0$, we infer that $ D \in \Yl0$, thus $D\in \Hm\cap\Yl0$.  Therefore,  $\Ul{-m}*\Yl0 = \Ul{-m}*(\Hm\cap\Yl0)$. The proof for the co-aisle is analogous.
	\end{proof}
	
	\begin{Lemma}\label{lemma5.2}
		Let $(\Xl0,\Xg0) \in \tstr[\clU[m],\clU]$ be a $t$-structure on $\D$, then 
		$$\psi'\circ\phi(\Xl0,\Xg0)=(\clS\cap\Xl0,\clS\cap\Xg0).$$
		\end{Lemma}
		\begin{proof}
		Since $(\Xl0,\Xg0) \in \tstr[\clU[m],\clU]$, we obtain an $s$-torsion pair $\phi(\Xl0,\Xg0)\in\stors\Hm$ by \Cref{proposition4.1}. By \Cref{proposition4.1} again, we have a $t$-structure
		\begin{align*}
			\psi'\circ\phi(\Xl0,\Xg0)&=\psi'(\Hm\cap\Xl0,\Hm[1]\cap\Xg0)\\&=((\clS\cap\Ul{-m})*(\Hm\cap\Xl0),(\Hm[1]\cap\Xg0)*(\clS\cap\Ug0)).
		\end{align*}
		
		It is clear that $(\clS\cap\Ul{-m})*(\Hm\cap\Xl0)\subseteq\clS\cap\Xl0$. Conversely, let $ X\in \clS\cap\Xl0$. There exists a distinguished triangle in $\clS$
		$$ A\rightarrow X \rightarrow B \rightarrow A[1]$$
		with $A \in \clS\cap\Ul{-m}$ and $ B \in \clS\cap\Ug{-m+1}$. We claim that $ B \in \Hm\cap\Xl0  $, thus $ X\in (\clS\cap\Ul{-m})*(\Hm\cap\Xl0)$.  In fact, since $ X \in \Xl0$ and $ A \in \clS\cap\Ul{-m} \subseteq \Ul{-m} \subseteq \Xl0$, it follows that $ B\in \Xl0\subseteq \Ul0$. We have $B \in \clS\cap\Ug{-m+1} \cap \Ul0=\Hm$. Therefore, $B \in (\Hm\cap\Xl0)$. The proof for the co-aisle is analogous.
	\end{proof}
	
	According to \Cref{proposition4.1}, \Cref{lemma5.1} and \Cref{lemma5.2}, we have the following result.
	
	\begin{Thm}\label{theorem5.3}
		Let $\D$ be a triangulated category and $\clU=(\Ul0,\Ug0)$ be a $t$-structure on $\D$ with heart $\clH$. Let $\clS$ be a triangulated full subcategory of $\D$ such that $\clU_\clS=(\clS\cap\Ul0,\clS\cap\Ug0)$ is a $t$-structure on $\clS$ and $\clH \subseteq \clS$. Then there are order preserving, mutually inverse bijections
		$$\xymatrix@R=10pt{\tstr[\clU[m],\clU] \ar@<0.5ex>[rr]^{\lambda} && \tstr[\clU_\clS[m],\clU_\clS] \ar@<0.5ex>[ll]^{\mu}	}$$	
		given by
		\begin{align}
			\lambda(\Xl0,\Xg0)&=( \clS\cap\Xl0,\clS\cap\Xg0), \label{eq5.1}\\ 
			\mu(\Yl0,\Yg0)&=(\Ul{-m}*\Yl0,\Yg0*\Ug0).\label{eq5.2}
		\end{align}
	\end{Thm}
	
	\begin{Ex}
		Let $\Db\A$ and $D(\A)$ be the bounded and unbounded derived categories of an abelian category $\A$, with $\Db\A$ being a triangulated full subcategory of $D(\A)$. Let $\clU=(\Ul0,\Ug0)$  be the standard $t$-structure on  $D(\A)$ with heart $\A$. Note that $\clU_{\Db\A}=(\Db\A\cap\Ul0,\Db\A\cap\Ug0)$ is the standard $t$-structure on $\Db\A$ with heart $\A$. By \Cref{theorem5.3}, there exists order preserving, mutually inverse bijections
		$$\xymatrix@R=10pt{\tstr[\clU[m],\clU] \ar@<0.5ex>[rr]^{\lambda\qquad} && \tstr[\clU_{\Db\A}[m],\clU_{\Db\A}] \ar@<0.5ex>[ll]^{\mu\qquad}}$$
		where $\lambda$ and $\mu$ are given by \eqref{eq5.1} and \eqref{eq5.2}.
	\end{Ex}
	
	\begin{Ex}[{\cite[Theorem 5.1]{CLZ23}}]
		Let $\D$ be a triangulated category with a $t$-structure $\clU=(\Ul0,\Ug0)$. Let $\clS$ be the smallest triangulated full subcategory containing the heart $\clH=\Ul0\cap\Ug0$. Then $\clU_{\clS}=(\clS\cap\Ul0,\clS\cap\Ug0)$ is a $t$-structure on $\clS$, and there are order preserving, mutually inverse bijections
		$$\xymatrix@R=10pt{\tstr[\clU[m],\clU] \ar@<0.5ex>[rr]^{\lambda} && \tstr[\clU_\clS[m],\clU_\clS] \ar@<0.5ex>[ll]^{\mu}	}$$	
		given by \eqref{eq5.1} and \eqref{eq5.2}.
	\end{Ex}

\end{document}